\documentclass[12pt, twoside, leqno]{article}
\usepackage{amsmath,amsthm,amssymb,mathtools,amsfonts,amscd}
\usepackage{hyperref}
\usepackage{enumitem}
\newtheorem{theorem}{Theorem}[section]
\newtheorem{lemma}[theorem]{Lemma}

\newtheorem{proposition}{Proposition}[section]
\newtheorem{corollary}{Corollary}[section]
\theoremstyle{definition}
\newtheorem{definition}{Definition}[section]
\newtheorem{remark}{Remark}[section]
\title{Various notions of topological transitivity in non-autonomous and generic dynamical systems}
\date{}

\author{P. Chiranjeevi\\Email: \href{mailto:chiranjeevi.perikala@uohyd.ac.in}{chiranjeevi.perikala@uohyd.ac.in} \\ School of Mathematics and Statistics\\ University of Hyderabad, India \and Rameshwari Gupta\\Email: \href{mailto:20mmpp04@uohyd.ac.in}{20mmpp04@uohyd.ac.in} \\ School of Mathematics and Statistics\\ University of Hyderabad, India}
\begin{document}
	\maketitle
	
	\begin{abstract}
	We consider two types of dynamical systems namely non-autonomous discrete dynamical systems(NDDS) and generic dynamical systems(GDS). In both of them, we study various notions of transitivity.  We give many equivalent conditions for each of these notions and present the implications among these in NDDS and GDS. For a given NDDS, we associate a GDS and discuss whether if the given NDDS has a particular variation of transitivity then the associated GDS also has such a variation and vice versa.
\end{abstract}

\textbf{keywords:}
	Non-autonomous discrete dynamical system,  transitivity, mixing, locally eventually onto, generic dynamical system.


\section{Introduction}

In this paper, we discuss various notions of transitivity for non-autonomous discrete dynamical systems (NDDS) and for generic dynamical systems (GDS). Our study is along the lines of that which is made for autonomous discrete dynamical systems (ADDS) in \cite{EAN} and for topological dynamical systems (TDS) in \cite{AN}. The systems ADDS, NDDS and TDS are well studied in the literature and in the present paper we introduce GDS. Throughout this paper, we assume that the underlying space is a compact metric space.

A \emph{non-autonomous discrete dynamical system (NDDS)} is a pair $(X,f_{1,\infty})$ where $X$ is a compact metric space and $f_{1,\infty}=(f_n)_{n\in \mathbb{N}}$ is a sequence of continuous self maps on $X$. For $n\in \mathbb{N}$, the composition 
\begin{center}
	$f_1 ^n:= f_n\circ f_{n-1}\circ \cdots \circ f_2\circ f_1$ 
\end{center}is said to be $n^{th}$- iterate of $f_{1,\infty}$.

We denote $(f_1 ^n)^{-1}$ by $f_1 ^{-n}$. i.e., $f_1 ^{-n}= f_1 ^{-1}\circ f_2 ^{-1}\circ \cdots \circ f_n ^{-1}$. \\
The case when $f_{1, \infty}$ is a constant sequence $(f)$, the pair $(X, f_{1, \infty})$ reduces to the classical discrete dynamical system $(X,f )$ which is also known as autonomous discrete dynamical system (ADDS).

Transitivity for NDDS generated by uniformly convergent sequences was well studied in \cite{SR}. Dynamics of non-autonomous systems generated by finite collection of maps was studied in \cite{SI} and \cite{RS}. The authors of \cite{JSC}, \cite{C} and \cite{Si} studied $\omega$-limit sets of non-autonomous discrete dynamical systems. In \cite{SK}, minimality for non-autonomous discrete dynamical systems was addressed.

In \cite{EAN}, the authors provided some equivalent conditions of some dynamical properties for autonomous discrete dynamical systems and such results for topological dynamical systems are in \cite{AN}. In the present paper, we investigate whether such results hold for non-autonomous discrete dynamical systems. For instance, it is proved in \cite{EAN} that an ADDS is topologically transitive if and only if there exists an element whose orbit is dense. In the present paper, we prove that a similar result holds for NDDS also. In other words, we prove that an NDDS is topologically transitive if and only if there exists an element with dense orbit. Since the orbit of an element in an NDDS need not be invariant, a generalization of the orbit, called extended orbit is given in \cite{N}. In this paper, we have introduced a new notion called extended transitivity which is proved to be equivalent to the existance of an element with dense extended orbit. For similar reason, notions like strong extended transitivity and extended minimality are introduced in this paper.

Since the extended transitivity of the system $(X,f_{1,\infty})$ does not depend on sequence structure of $f_{1,\infty}$, instead of saying that $(X,f_{1,\infty})$ is extended transitive, we may say that $(X,\mathcal F)$ is transitive where $\mathcal{F}=\{f_n: n\in\mathbb{N}\}$. This motivates us to consider dynamical systems where $\mathcal{F}$ need not be countable. We call this type of dynamical systems as \emph{generic dynamical systems}. 

Two dynamical systems are identified to be the same through a notion called conjugacy. In the case of autonomous systems, a \emph{topological conjugacy} from $(X,f)$ to $(Y,g)$ is a homeomorphism $\pi : X\rightarrow Y$ such that $\pi \circ f=g\circ\pi$. It follows that $\pi \circ f^n=g^n\circ\pi$ for every $n\in \mathbb{N}$. However, in the case of non-autonomous system $(X,f_{1,\infty})$ and $(Y,g_{1,\infty})$, we need to ensure this as a part of the definition. This can be ensured by assuming that $\pi$ satisfies the condition either $\pi \circ f_1^n =g_1^n \circ\pi $ for every $n\in \mathbb{N}$ or $\pi \circ f_n =g_n \circ \pi $ for every $n\in \mathbb{N}$. If $\pi \circ f_1^n =g_1^n \circ\pi $ for every $n\in \mathbb{N}$, we call $\pi $ a \emph{conjugacy} and if  $\pi \circ f_n =g_n \circ \pi $ for every $n\in \mathbb{N}$, we call $\pi $ a \emph{strong conjugacy}. Notions like toplogical transitivity are preseved under conjugacy and the notions like extended transitivity are preseved under strong conjugacy. Notion of conjugacy  from a GDS $(X,\mathcal{F})$ to $(Y,\mathcal{G})$ is also introduced in this paper. Moreover, we introduce and study the notion of strong conjugacy $(X,\mathcal{F})$ to $(Y,\mathcal{G})$ when $\mathcal{F}, \mathcal{G}$ are assumed to be topolgocal spaces. See Section \ref{conju} for results about conjugacy of non-autonomous systems and Theorems \ref{conjugacy}, \ref{strongconjugacy} for results about conjugacy of generic systems.

Given an NDDS $(X,f_{1,\infty})$ such that $f_m\circ f_n=f_n\circ f_m$ for all $m,n\in\mathbb{N}$. In Theorem \ref{rearrangement1} and \ref{rearrangement2}, we prove that if $(X,f_{1,\infty})$ either topologically transitive or mixing, then $(X,g_{1,\infty})$ also exhibits identical notion for every finite rearrangement $g_{1,\infty}$ of $f_{1,\infty}$. See Theorem \ref{rearrangement3} for a similar result about locally eventually onto systems.

In the last part of the paper, for a given NDDS $(X,f_{1,\infty})$, we associate a suitable GDS $(X,\mathcal{F})$ and discuss whether $(X,f_{1,\infty})$ has a particular variation of transitivity whenever $(X,\mathcal{F})$ also has such a variation. In Theorem \ref{4.19}, we see that $\mathcal{F}=\{f_n: n\in \mathbb{N}\}$ is the appropreate choice for some notions like exetended transitivity. For some notions like very strong transitivity, we see in Theorems \ref{4.20} and \ref{4.21} that $\mathcal{F}=\{f_1^n:n\in\mathbb{N}\}$ is the appropreate choice.

The following is the list of various notions of transitivity that we study in this paper:
\begin{enumerate}[label=\upshape(\roman*), leftmargin=*, widest=iii]
	\item topological transitivity (TT)
	\item extended transitivity
	\item strong transitivity (ST)
	\item strong extended transitivity
	\item very strong transitivity (VST)
	\item extended minimality 
	\item exact transitivity 
	\item strong exact transitivity 
	\item topological mixing (TM)
	\item locally eventually onto (LEO)
\end{enumerate}

For literature on topological transitivity and mixing for non-autonomous discrete dynamical systems see \cite{PK}, \cite{LY}, \cite{MA}, \cite{SD}, \cite{Si}.

\section{Preliminaries}
We begin with some elementary results.\\ 
For a map $f:X\rightarrow Y$ with $A\subset X, B\subset Y$
\begin{align}
	\label{2.1} f(A)\cap B\neq \emptyset\Leftrightarrow A\cap f^{-1}(B)\neq \emptyset \\ 
	\label{2.2}	f(A)\subset B\Leftrightarrow A\subset f^{-1}(B)\Leftrightarrow f^{-1}(B^c)\subset A^c \\ 
	\label{2.3}	\textit{ If $f$ is surjective, then $f^{-1}(B)\subset A\Rightarrow B\subset f(A)$.} 
\end{align}
Now we give some basic definitions and then provide some results related to these definitions.
\begin{definition} \normalfont
	Let $(X,f_{1,\infty})$ be an NDDS.
	\begin{enumerate}[label=\upshape(\roman*), leftmargin=*, widest=iii]
		\item The \emph{orbit} of $x\in X$ is $O(x,f_{1,\infty}) =\{f_1 ^n (x):n\in {\mathbb{N}}\}$.
		For $N\in \mathbb{N}$,
		$O_N(x,f_{1,\infty}) =\{f_1 ^n (x):1\leq n\leq N\}$ is said to be a \emph{partial orbit} of $x$.
		\item The \emph{negative orbit} of $x\in X$ is\\
		$O^- (x,f_{1,\infty})=\{y\in X : f_1 ^{n} (y)=x$ for some $ n\in {\mathbb{N}}\} = \bigcup\limits _{n\in \mathbb{N}} f_1 ^{-n} (x)$.\\
		For $N\in \mathbb{N}$, $O_N^-(x,f_{1,\infty}) = \bigcup\limits _{1\leq n\leq
			N} f_1 ^{-n} (x)$ is said to be a \emph{partial negative orbit} of $x$.
		\item The set of limit points of an orbit $O(x,f_{1,\infty})$ is the \emph{$\omega $-limit set} of $x$, which is denoted as $\omega(x,f_{1,\infty})$, i.e., $\omega(x,f_{1,\infty})=\bigcap\limits_{N\in \mathbb{N}}\overline{\{f_1 ^n (x):n\geq N\}}$, and $O(x,f_{1,\infty})\cup \omega(x,f_{1,\infty})= \overline{O(x,f_{1,\infty})}$ is the orbit closure of $x$.
		\item A point $x\in X$ is called \emph{recurrent} when $x\in \omega(x,f_{1,\infty})$.
	\end{enumerate}
\end{definition}	

\begin{definition} \normalfont
	Let $(X,f_{1,\infty})$ be an NDDS and $A$ be a subset of $X$. Then $A$ is said to be
	\begin{enumerate}[label=\upshape(\roman*), leftmargin=*, widest=iii]
		\item a \emph{+ invariant set} if $f_1 ^n (A)\subset A$ for all $n\in \mathbb{N}$.
		\item a \emph{strong + invariant set} if $f_n (A)\subset A$ for all $n\in \mathbb{N}$.
		\item a \emph{$-$ invariant set} if $f_1 ^{-n} (A)\subset A$ for all $n\in \mathbb{N}$.
		\item a \emph{strong $-$ invariant set} if $f_n ^{-1}(A)\subset A$ for all $n\in \mathbb{N}$.
		\item a \emph{weakly $-$ invariant set} if $A\subset f_1 ^n(A)$ for all $n\in \mathbb{N}$.
		\item an \emph{extended $-$ invariant set} if $A\subset f_n(A)$ for all $n\in \mathbb{N}$.
		\item an \emph{invariant set} if $f_1 ^n (A)=A$ for all $n\in \mathbb{N}$.
	\end{enumerate}
\end{definition}
It is easy to see that $A$ is invariant if and only if $f_n(A)=A$ for all $n\in \mathbb{N}$.

We now prove that the $-$ invariant sets are precisely the complements of + invariant sets.
\begin{lemma}\label{complement}
	For an NDDS $(X,f_{1,\infty})$, a subset $A\subset X$ is + invariant if and only if $A^c$ is $-$ invariant.
\end{lemma}
\begin{proof}
	Assume $A$ is + invariant, i.e., $f_1^n(A)\subset A$ and so $f_1^{-n}(A^c) \subset A^c$ for all $n\in \mathbb{N}$ by equation (\ref{2.2}). Hence $A^c$ is $-$ invariant. Now suppose that $A^c$ is $-$ invariant. Thus, $f_1^{-n}(A^c)\subset A^c$ and so $f_1^n(A)\subset A$ for all $n\in \mathbb{N}$ by equation (\ref{2.2}). Therefore $A$ is + invariant.
\end{proof}
The next lemma says that, the strong $-$ invariant sets are precisely the complements of strong + invariant sets.
\begin{lemma}
	For an NDDS $(X,f_{1,\infty})$, a subset $A\subset X$ is strong + invariant if and only if $A^c$ is strong $-$ invariant.
\end{lemma}
\begin{proof}
	The proof is similar to that of Lemma \ref{complement}. We omit the details.
\end{proof}
The next lemma gives an equivalent condition for a set to be weakly $-$ invariant.
\begin{lemma}\label{3}
	For an NDDS $(X,f_{1,\infty})$, a subset $A\subset X$ is weakly $-$ invariant if and only if $f_1 ^n (A\cap f_1 ^{-n}(A))=A$, for all $n\in \mathbb{N}$.
\end{lemma}
\begin{proof}
	We prove the necessity part and the sufficiency part is trivial. To prove the necessity part, assume $A$ is weakly $-$ invariant and fix $n\in \mathbb{N}$. Then $A\subset f_1^n(A)$ and so $f_1^n(A\cap f_1^{-n}(A))\subset f_1^n(A)\cap A=A$. Fix $a\in A$. Since $A\subset f_1^n(A)$, there exists $b\in A$ such that $f_1^n(b)=a$. Then $b\in f_1^{-n}(A)$ which implies that $b\in A\cap f_1^{-n}(A)$. As $a=f_1^n(b)\in f_1^n(A\cap f_1^{-n}(A))$, we get $A\subset  f_1^n(A\cap f_1^{-n}(A))$. Hence $A=  f_1^n(A\cap f_1^{-n}(A))$, for all $n\in \mathbb{N}$. 
\end{proof}
The next result gives an equivalent condition for a set to be extended $-$ invariant.
\begin{lemma}
	For an NDDS $(X,f_{1,\infty})$, a subset $A\subset X$ is extended $-$ invariant if and only if $f_n (A\cap f_n ^{-1}(A))=A$ for all $n\in \mathbb{N}$.
\end{lemma}
\begin{proof}
	By replacing $f_1^n$ with $f_n$ in the proof of Lemma \ref{3}, we can easily obtain the proof the above lemma.
\end{proof}

\begin{remark}
	In an NDDS $(X,f_{1,\infty})$,
	\begin{enumerate}[label=\upshape(\roman*), leftmargin=*, widest=iii]
		\item the orbit of an element need not be + invariant.
		\item the negative orbit of an element need not be $-$ invariant.
		\item a subset $A\subset X$ is + invariant if and only if  $O(x,{f_{1,\infty}})\subset A$ for every $x\in A$.
		\item a subset $A\subset X$ is $-$ invariant if and only if  $O^-(x,{f_{1,\infty}})\subset A$ for every $x\in A$.
		\item the $\omega$-limit set of an element need not be + invariant. (For examples see \cite{Si}).
	\end{enumerate}
\end{remark}
Let us now consider the set of words of finite lenght with alphabet $\mathbb{N}$: $\sum:=\{\alpha=(n_1, n_2,\dots, n_p):n_i \in \mathbb{N}, p\in \mathbb{N}\}$, and for a given sequence $(f_n)$ of continuous self maps on $X$ and $\alpha=(n_1, n_2,\dots, n_p)\in\sum $, let us denote $f_{\alpha}:= f_{n_p}\circ f_{n_{p-1}}\circ\cdots \circ f_{n_1}$.
\begin{definition} \normalfont
	The \emph{extended orbit} of $y\in X$ is $J(y)=\{f_{\alpha}(y):\alpha\in\sum\}$. For any nonempty subset $\sum'\subset \sum$, we write $J_{\sum'}(y)=\{f_{\alpha}(y):\alpha\in\sum'\}$.\\
	The set of limit points of the extended orbit $J(x)$ of an element $x\in X$ is called the \emph{extended $\omega$-limit set} of $x$, and it is denoted by  $\omega_e(x,f_{1,\infty})$. The set $J(x)\cup \omega_e(x,f_{1,\infty})= \overline{J(x)}$ is the extended orbit closure of $x$.
\end{definition}
\begin{definition} \normalfont
	The \emph{extended negative orbit} of $y\in X$ is\\
	$J^- (y)=\{x\in X:f_{\alpha}(x)=y$, for some $ {\alpha}\in \sum\}$. For any nonempty subset $\sum'\subset \sum$, we write $J_{\sum'}^-(y)=\{x\in X:f_{\alpha}(x)=y,$ for some $\alpha\in\sum'\}$.
\end{definition}
In the case of ADDS, the orbit of an element is + invariant but in the case of NDDS, the orbit $O(x,f_{1,\infty})$ need not be + invariant. We will see in the following proposition that the extended orbit of every element is strong + invariant. 
\begin{proposition}\label{P1}
	The extended orbit $J(y)$ of any $y\in X$ has the following properties:
	\begin{enumerate}[label=\upshape(\roman*), leftmargin=*, widest=iii]
		\item $O(y,f_{1,\infty})\subset J(y)$.
		\item The set $J(y)$ is strong + invariant.
		\item A subset $A\subset X$ is strong + invariant if and only if $J(y)\subset A$ for every $y\in A$.
	\end{enumerate}
\end{proposition}
\begin{proof}
	For the proof we refer the reader to \cite{N}.
\end{proof}
Again in an NDDS, the negative orbit $O^-(x,f_{1,\infty})$ need not be $-$ invariant. We will see in the following result that the extended negative orbit is strong $-$ invariant. 
\begin{proposition}\label{P2}
	The extended negative orbit $J^-(y)$ of any $y\in X$ has the following properties:
	\begin{enumerate}[label=\upshape(\roman*), leftmargin=*, widest=iii]
		\item $O^-(y,f_{1,\infty})\subset J^-(y)$.
		\item The set $J^-(y)$ is strong $-$ invariant.
		\item A subset $A\subset X$ is strong $-$ invariant if and only if  $J^-(y)\subset A$ for every $y\in A$.
	\end{enumerate}
\end{proposition}
\begin{proof}
	The proof is straightforward from the definition of extended negative orbit.
\end{proof}

In an autonomous discrete dynamical system $(X,f)$, if $f$ is surjective and $A$ is either invariant or weakly $-$ invariant subset of $X$,  then $\overline{A}$ satisfies the corresponding property. If $f$ is open surjective map then interior of any + invariant subset of $X$ is also + invariant and closure of any $-$ invariant subset of $X$ is $-$ invariant. For the proof of these results see \cite{EAN}. It is natural to ask whether similar results hold in the case of NDDS.
The following two lemmas give an affirmative answer. 

\begin{lemma}\label{5}
	If $f_n$ is surjective for every $n\in \mathbb{N}$ and $A$ is a subset of $X$, then the following statements hold:
	\begin{enumerate}[label=\upshape(\roman*), leftmargin=*, widest=iii]
		\item If $A$ is $-$ invariant, then $A$ is weakly $-$ invariant.
		\item  The set $A$ is invariant if and only if $A$ is both + invariant and weakly $-$ invariant. Also, $A=f_1^{-n}(A)$ for all $n\in \mathbb{N}$ if and only if $A$ is both + invariant and $-$ invariant, in that case, it is invariant.
		\item If $A$ is either + invariant, weakly $-$ invariant or invariant, then $\overline A$ satisfies the corresponding property.
		\item If $f_n$ is open map for every $n\in \mathbb{N}$ and if $A$ is + invariant then $A^{\circ} $ is + invariant. Also, if $A$ is $-$ invariant then $\overline A$ is $-$ invariant.
	\end{enumerate}
\end{lemma}
\begin{proof}
	(i) follows from equation (\ref{2.3}).\\
	(ii) follows from equations (\ref{2.2}) and (\ref{2.3}).\\
	Proof of (iii): If $A$ is + invariant, then $f_1^n(A)\subset A$ and so $f_1^n(\overline A) \subset \overline {f_1^n(A)} \subset \overline{A}$, for every $n\in \mathbb{N}$. Therefore $\overline{A}$ is + invariant.\\
	Suppose $A$ is weakly $-$ invariant, and let $n\in \mathbb{N}$. Then $A\subset f_1^n(A)$ or $\overline{A}\subset \overline{f_1^n(A)}$. Continuity of $f_1^n$ gives us $\overline{f_1^n(A)}\subset  f_1^n(\overline{A})$. The reverse inclusion is always true. Hence we have $\overline{f_1^n(A)}=  f_1^n(\overline{A})$. Thus, we have $\overline A\subset  f_1^n(\overline{A})$ or $\overline{A}$ is weakly $-$ invariant. Hence $\overline{A}$ is invariant if $A$ is invariant.\\
	Proof of (iv): Assume $f_n$ is an open map for every $n\in \mathbb{N}$. Since $A^{\circ}$ is open, $f_n(A^{\circ})$ is open for all $n\in \mathbb{N}$. Since $A$ is + invariant and $f_1^n$ is continuous, we have $f_1^n(A^{\circ})\subset (f_1^n(A))^{\circ}\subset A^{\circ}$ for all $n\in \mathbb{N}$. Thus, $A^{\circ}$ is + invariant. If $A$ is $-$ invariant, then $A^c$ is + invariant. By above, $(A^c)^{\circ}$ is + invariant, hence $\overline A$= $((A^c)^{\circ})^c$ is $-$ invariant.
\end{proof}

\begin{lemma}\label{6}
	If $f_n$ is surjective for every $n\in \mathbb{N}$ and $A$ is a subset of $X$, then the following statements hold:
	\begin{enumerate}[label=\upshape(\roman*), leftmargin=*, widest=iii]
		\item If $A$ is strong $-$ invariant, then $A$ is extended $-$ invariant.
		\item The set $A$ is invariant if and only if $A$ is both strong + invariant and extended $-$ invariant. Also, $A=f_n^{-1}(A)$ for all $n\in \mathbb{N}$ if and only if $A$ is both strong + invariant and strong $-$ invariant, in that case, it is invariant.
		\item If $A$ is either strong + invariant, extended $-$ invariant or strong invariant, then $\overline A$ satisfies the corresponding property.
		\item If $f_n$ is open map for every $n\in \mathbb{N}$ and if $A$ is strong + invariant then $A^{\circ} $ is strong + invariant. Also, if $A$ is strong $-$ invariant then $\overline A$ is strong $-$ invariant.
	\end{enumerate}
\end{lemma}
\begin{proof}
	The proof is similar to the proof of Lemma \ref{5}.
\end{proof}

\begin{definition} \normalfont
	Let $\epsilon>0$. A subset $U\subset X$ is called \emph{$\epsilon$ dense} if $U$ intersects every open ball of radius $\epsilon$ in $X$.
\end{definition}
Some notions of topological transitivity can be written in terms of $\epsilon$ dense sets. To establish such results we recall the following result from \cite{EAN}.
\begin{lemma}\label{7}
	Let $(U_n)$ be a sequence of subsets of $X$. Then
	\begin{enumerate}[label=\upshape(\roman*), leftmargin=*, widest=iii]
		\item $\bigcup\limits_{n=1}^{\infty} U_n$ is dense if and only if for every $\epsilon>0$, there exists $N\in \mathbb{N}$ such that $\bigcup\limits_{n=1}^{N} U_n$ is $\epsilon$ dense.
		\item If each $U_n$ is open and $\bigcup\limits_{n=1}^{\infty} U_n = X$, then there exists $N\in \mathbb{N}$ such that $\bigcup\limits_{n=1}^{N} U_n =X$.
	\end{enumerate}
\end{lemma}

\section{Main Results}
\subsection{Topological transitivity}
For any two nonempty subsets $U,V\subset X$, we define the \emph{set of hitting times} $N(U,V)$ as:\\
$N(U,V)=\{n\in \mathbb{N}:f_1^n(U)\cap V \neq \emptyset\}$
$=\{n\in \mathbb{N}:U\cap f_1^{-n}( V) \neq \emptyset\}$ and the \emph{extended set of hitting times} $N_e(U,V)$ as:\\
$N_e(U,V)=\{\alpha\in \sum:f_{\alpha}(U)\cap V \neq \emptyset\}$
$=\{\alpha\in \sum:U\cap f_{\alpha}^{-1}( V) \neq \emptyset\}$.\\
Clearly, $N(U,V) \subset N_e(U,V)$.\\
For an element $x\in X$, we write $N(U,\{x\})$ as $N(U,x)$ and $N_e(U,\{x\})$ as $N_e(U,x)$.

\begin{definition}[Topological transitivity] \normalfont
	A system $(X,f_{1,\infty})$ is called \emph{topologically transitive} if for every opene(=nonempty open) $V\subset X$, $\bigcup\limits_{n=1}^{\infty}f_1^n(V)$ is dense in $X$. Equivalently, if for every opene pair $U,V\subset X$, the set of hitting times $N(U,V)$ is nonempty. i.e., for every opene pair $U,V\subset X$, there exists $n\in \mathbb{N}$ such that $f_1^n(U)\cap V \neq \emptyset$.\\
	
	A point $x\in X$ is called a \emph{transitive point} if for every opene $V\subset X$, the set of hitting times $N(x,V)$ is nonempty. Equivalently, if the orbit $O(x,f_{1,\infty})=\{f_1^n(x):n\in \mathbb{N}\}$ is dense in $X$. The set of all transitive points is denoted by $T(f_{1,\infty})$.
\end{definition}
In Theorem \ref{8} we will prove that an NDDS $(X,f_{1,\infty})$ is topologically transitive if and only if the set of transitive points is a $G_{\delta}$ subset of $X$. To prove these implications we need the following proposition. 

\begin{proposition}\label{P3}
	Let $(X,f_{1,\infty})$ be an NDDS. Then the set $T(f_{1,\infty})$ of transitive points equals $\{x:\omega(x,f_{1,\infty})=X\}$.
\end{proposition}
\begin{proof}
	Let $x\in X$ be such that $\omega(x,f_{1,\infty})=X$. Then $ \overline{O(x,f_{1,\infty})}=O(x,f_{1,\infty})\cup \omega(x,f_{1,\infty})=X$, which implies $x\in T(f_{1,\infty})$.
	Conversely, assume $x\in T(f_{1,\infty})$. Let $y\in X$ and $U_y$ be any neighbourhood of $y$. Since $O(x,f_{1,\infty})$ is dense in $X$, there exists a sequence in the orbit of $x$ converging to $y$. This means that $y$ is an $\omega$-limit point of $x$, and $\omega(x,f_{1,\infty})=X$.
\end{proof}
Now we give the definition of extended transitivity which is a generalization of topological transitivity.
\begin{definition}[Extended transitivity] \normalfont
	A system $(X,f_{1,\infty})$ is called \emph{extended transitive} if for every opene $V\subset X$, $\bigcup\limits_{\alpha\in\sum}f_{\alpha}(V)$ is dense in $X$. Equivalently, if for every opene pair $U,V\subset X$, the extended set of hitting times $N_e(U,V)$ is nonempty.\\
	A point $x\in X$ is called an \emph{extended transitive point} if for every opene $V\subset X$, the extended set of hitting times $N_e(x,V)$ is nonempty. Equivalently, if the extended orbit $J(x)=\{f_{\alpha}(x):\alpha\in \sum\}$ is dense in $X$. The set of all extended transitive points is denoted by $T_e(f_{1,\infty})$.
\end{definition}
We will use the following proposition in Theorem \ref{9}.
\begin{proposition}\label{P4}
	Let $(X,f_{1,\infty})$ be an NDDS. Then the set $T_e(f_{1,\infty})$ of extended transitive points equals $\{x:\omega_e(x,f_{1,\infty})=X\}$.
\end{proposition}
\begin{proof}
	The proof is similar to the proof of Proposition \ref{P3}.
\end{proof}
%
%
\begin{remark}
	Every transitive point is extended transitive and recurrent.
	On the other hand one can easily observe that every topologically transitive system is extended transitive.
\end{remark}

The following theorem gives some equivalent conditions for topological transitivity.
\begin{theorem}\label{8}
	Let $(X,f_{1,\infty})$ be an NDDS, where $X$ is a perfect space(i.e., has no isolated points). Then the following are equivalent:
	\begin{enumerate}[label=\upshape(\roman*), leftmargin=*, widest=iii]
		\item  The system is topologically transitive.
		\item  For every pair of opene sets $U$ and $V$ in $X$, there exists $n\in \mathbb{N}$ such that $f_1^{-n}(U)\cap V\neq\emptyset$.
		\item  For every pair of opene sets $U, V$ in $X$ the set $N(U,V)$ is nonempty.
		\item  For every pair of opene sets $U, V$ in $X$ the set $N(U,V)$ is infinite.
		\item  There exists $x\in X$ such that the orbit $O(x,f_{1,\infty})$ is dense in $X$.
		\item  The set $T(f_{1,\infty})$ of transitive points is a dense, $G_{\delta}$ subset of $X$.
		\item  For every opene set $U\subset X$, $\bigcup\limits_{n=1}^{\infty}f_1^n(U)$ is dense in $X$.
		\item  For every opene set $U\subset X$, and $\epsilon>0$, there exists $N\in \mathbb{N}$ such that  $\bigcup\limits_{n=1}^{N}f_1^n(U)$ is $\epsilon$ dense in $X$.
		\item  For every opene set $U\subset X$, $\bigcup\limits_{n=1}^{\infty}f_1^{-n}(U)$ is dense in $X$.
		\item For every opene set $U\subset X$, and $\epsilon>0$, there exists $N\in \mathbb{N}$ such that  $\bigcup\limits_{n=1}^{N}f_1^{-n}(U)$ is $\epsilon$ dense in $X$.
	\end{enumerate}
\end{theorem}
\begin{proof}
	Notice that condition (vii) is the definition of topological transitivity and so, of course, (i)$\Leftrightarrow$(vii). \\
	(ii)$\Leftrightarrow$(iii), (iii)$\Leftrightarrow$(vii) and (iii)$\Leftrightarrow$(ix) follow from the definition of $N(U,V)$.
	(vii)$\Leftrightarrow$(viii), and (ix)$\Leftrightarrow$(x) follow from Lemma \ref{7}.\\
	So for, we have established equivalence of (i), (ii), (iii), (vii), (viii), (ix), (x). To complete the proof we prove that (ix)$\Rightarrow$(vi)$\Rightarrow$(v)$\Rightarrow$(iv)$\Rightarrow$(iii).\\
	To prove (ix)$\Rightarrow$(vi): 
	Let $x\in T(f_{1,\infty})$, then for all opene $U\subset X$, there exists $n\in \mathbb{N}$ such that $f_1^n(x)\in U$. So $x\in \bigcup\limits_{n=1}^{\infty} f_1^{-n}(U)$ for all opene $U$. Since every compact metric space is second countable, we have a countable base $\mathcal{B}$ and so $T(f_{1,\infty})\subset\bigcap\limits_{U\in \mathcal{B}} \bigcup\limits_{n=1}^{\infty} f_1^{-n}(U)$.\\
	Now let $U\in \mathcal{B}$ and $x\in\bigcup\limits_{n=1}^{\infty} f_1^{-n}(U)$. This implies that $x\in \bigcup\limits_{n=1}^{\infty} f_1^{-n}(U)$, and hence $x\in f_1^{-n}(U)$ for some $n\in \mathbb{N}$. So $f_1^n(x)\in U$ for some $n\in \mathbb{N}$, which implies that $N(U,x)\neq \emptyset$, and hence $x\in T(f_{1,\infty})$. Therefore $\bigcap\limits_ {U\in \mathcal{B}}(\bigcup\limits_{n=1}^{\infty} f_1^{-n}(U))\subset T(f_{1,\infty}) $. Hence $\bigcap\limits_ {U\in \mathcal{B}}(\bigcup\limits_{n=1}^{\infty} f_1^{-n}(U))= T(f_{1,\infty}) $. Now by (ix), $\bigcup\limits_{n=1}^{\infty} f_1^{-n}(U)$ is dense in $X$ for all opene $U\subset X$. Thus, $T(f_{1,\infty}) $ is countable intersection of opene dense subsets in $X$. Since $X$ is a compact metric space, $T(f_{1,\infty}) $ is a dense, $G_{\delta}$ subset of $X$.\\
	(vi)$\Rightarrow$(v) follows by Proposition \ref{P3}.\\
	To prove (v)$\Rightarrow$(iv): If $X$ is perfect and $x\in T(f_{1,\infty}) $ then the orbit $O(x,f_{1,\infty})$ of $x$ intersects every opene set in an infinite set. Thus, $N(U,V)$ is infinite for every opene $U,V\subset X$.\\
	To conclude the proof, observe that (iv)$\Rightarrow$(iii) is trivial.
\end{proof}

\begin{corollary}
	Let $(X,f_{1,\infty})$ be a topologically transitive system. Then: 
	\begin{enumerate}[label=\upshape(\roman*), leftmargin=*, widest=iii]
		\item any opene, $-$ invariant subset $U\subset X$ is dense in $X$. 
		\item if $A\subset X$ is closed, + invariant, then either $A=X$ or $A$ is nowhere dense in $X$. 
	\end{enumerate}
\end{corollary}
\begin{proof}
	We first prove (i). Let the system $(X,f_{1,\infty})$ be topologically transitive, then by condition (ix) of Theorem \ref{8}, we have for every opene $U\subset X$, $\bigcup\limits_{n=1}^{\infty}f_1^{-n}(U)$ is dense in $X$. Suppose $U\subset X$ is opene and $-$ invariant, then $f_1^{-n}(U)\subset U$ for all $n\in \mathbb{N}$. Therefore $\bigcup \limits_{n=1}^{\infty}f_1^{-n}(U)\subset U$ and hence $U$ is dense in $X$. From Lemma \ref{complement}, both (i) and (ii) are equivalent.
\end{proof}
In the following theorem we provide some equivalent conditions for extended transitivity.

\begin{theorem}\label{9}
	Let $(X,f_{1,\infty})$ be an NDDS, where $X$ is a perfect space and each $f_n$ is surjective. Then the following are equivalent:
	\begin{enumerate}[label=\upshape(\roman*), leftmargin=*, widest=iii]
		\item  The system is extended transitive.
		\item  For every pair of opene sets $U$ and $V$ in $X$, there exists $\alpha\in \sum$ such that $f_{\alpha}^{-1}(U)\cap V\neq
		\emptyset$.
		\item  For every pair of opene sets $U, V$ in $X$ the set $N_e(U,V)$ is nonempty.
		\item  For every pair of opene sets $U, V$ in $X$ the set $N_e(U,V)$ is infinite.
		\item  There exists $x\in X$ such that the extended orbit $J(x)$ is dense in $X$.
		\item  The set $T_e(f_{1,\infty})$ of extended transitive points is a dense, $G_{\delta}$ subset of $X$.
		\item  For every opene set $U\subset X$, $\bigcup\limits_{\alpha\in\sum}f_{\alpha}(U)$ is dense in $X$.
		\item  For every opene set $U\subset X$, and $\epsilon>0$, there exists a finite subset $\sum'\subset \sum$ such that  $\bigcup\limits_{\alpha\in\sum'}f_{\alpha}(U)$ is $\epsilon$ dense in $X$.
		\item  For every opene set $U\subset X$, $\bigcup\limits_{\alpha\in\sum}f_{\alpha}^{-1}(U)$ is dense in $X$.
		\item  For every opene set $U\subset X$, and $\epsilon>0$, there exists a finite subset $\sum'\subset \sum$ such that  $\bigcup\limits_{\alpha\in\sum'}f_{\alpha}^{-1}(U)$ is $\epsilon$ dense in $X$.
		\item If $U\subset X$ is opene and strong $-$ invariant, then $U$ is dense in $X$.
		\item If $A\subset X$ is closed and strong + invariant, then either $A=X$ or $A$ is nowhere dense in $X$.
	\end{enumerate}
\end{theorem}
\begin{proof}
	We prove (ix)$\Leftrightarrow$(xi) and (xi)$\Leftrightarrow$(xii). The proof of the equivalence of first 10 conditions is similar to the proof of Theorem \ref{8}.\\
	(ix)$\Leftrightarrow$(xi): First we prove (ix)$\Rightarrow$(xi). Suppose $U\subset X$ is opene and strong $-$ invariant, then $f_n^{-1}(U)\subset U$ for all $n\in \mathbb{N}$. Therefore $f_{\alpha}^{-1}(U)\subset U$ for all $\alpha\in \sum$. Hence $\bigcup \limits_{\alpha\in\sum}f_{\alpha}^{-1}(U)\subset U$. By condition (ix), $\bigcup \limits_{\alpha\in\sum}f_{\alpha}^{-1}(U)\subset U$ which implies that $U$ is dense in $X$.\\
	For (xi)$\Rightarrow$(ix): We see that for any opene subset $U\subset X$, $\bigcup\limits_{\alpha\in\sum}f_{\alpha}^{-1}(U)$ is opene and strong $-$ invariant subset of $X$. Hence, $\bigcup\limits_{\alpha\in\sum}f_{\alpha}^{-1}(U)$ is dense in $X$  by (xi).\\
	(xi)$\Leftrightarrow$(xii): First we prove (xii)$\Rightarrow$(xi). Suppose $U\subset X$ is opene and strong $-$ invariant, then $U^c$ is closed and strong + invariant. Then by condition (xii), either $U^c=X$ or $U^c$ is nowhere dense in $X$. If $U^c=X$ then $U=\emptyset$ which is not possible. Therefore, $U^c$ is nowhere dense in $X$ which implies $U$ is dense in $X$.\\
	For (xi)$\Rightarrow$(xii): Suppose $A$ is closed and strong + invariant, then $A^c$ is open and $-$ invariant. If $A^c$ is empty then $A=X$. If $A^c\neq \emptyset$ then by condition (xi), $A^c$ is dense in $X$. Hence, $A$ is nowhere dense in $X$ since the complement of a dense open subset is nowhere dense.
\end{proof}
\subsection{Strong transitivity and very strong transitivity}
In this section, we define some notions which are strengthen to topological transitivity.
\begin{definition}[Strongly transitive] \normalfont
	A system $(X,f_{1,\infty})$ is called \emph{strongly transitive} if for every opene subset $V\subset X$, $\bigcup\limits_{n=1}^{\infty}f_1^n(V)=X$.
\end{definition}
\begin{definition}[Strongly extended transitive] \normalfont
	A system $(X,f_{1,\infty})$ is called \emph{strongly extended transitive} if for every opene subset $V\subset X$, $\bigcup\limits_{\alpha\in\sum}f_{\alpha}(V)=X$.
\end{definition}

\begin{definition}[Very strongly transitive] \normalfont
	A system $(X,f_{1,\infty})$ is called \emph{very strongly transitive} if for every opene subset $V\subset X$, there exists $k\in \mathbb{N}$ such that  $\bigcup\limits_{n=1}^{k}f_1^n(V)=X$.
\end{definition}
From the definitions it is easy to observe that
\begin{enumerate}[label=\upshape(\roman*), leftmargin=*, widest=iii]
	\item Very strong transitivity $\Rightarrow$ strong transitivity $\Rightarrow$ topological transitivity.
	\item Strong transitivity $\Rightarrow$ strong extended transitivity $\Rightarrow$ extended transitivity.
\end{enumerate}
We now provide some equivalent conditions for strong transitivity in Theorem \ref{10}. We will give such equivalent conditions for strong extended transitivity in Theorem \ref{11} and for very strong transitivity in Theorem \ref{12}.

\begin{theorem}\label{10}
	For an NDDS $(X,f_{1,\infty})$ the following are equivalent:
	\begin{enumerate}[label=\upshape(\roman*), leftmargin=*, widest=iii]
		\item  The system $(X,f_{1,\infty})$ is strongly transitive.
		\item  For every opene set $U\subset X$ and every point $x\in X$, there exists $n\in \mathbb{N}$ such that $x\in f_1^n(U)$.
		\item  For every opene set $U\subset X$ and every point $x\in X$, the set $N(U,x)$ is nonempty.
		\item  For every $x\in X$, the negative orbit $ O^-(x,f_{1,\infty})$ is dense in $X$.
		\item  For every $x\in X$, and $\epsilon>0$, there exists $N\in \mathbb{N}$ such that $O_N^-(x,f_{1,\infty})$ is $\epsilon$ dense in $X$.
	\end{enumerate}
\end{theorem}
\begin{proof}
	It is straight forward to observe (i)$\Leftrightarrow$(ii) from the definition of strong transitivity.\\
	(ii)$\Leftrightarrow$(iii) follows from the definition of $N(U,x)$.\\
	(ii)$\Leftrightarrow$(iv): Assume (ii) and suppose $O^-(x,f_{1,\infty})$ is not dense in $X$. Then there exists opene subset $V\subset X$ such that $\bigcup\limits_{n\in \mathbb{N}}f_1^{-n}(x)\cap V=\emptyset$ which implies that $f_1^{-n}(x)\cap V=\emptyset$, for all $n\in \mathbb{N}$. i.e., $x\notin f_1^n(V)$ for any $n\in \mathbb{N}$. This contradicts condition (ii). Hence, $O^-(x,f_{1,\infty})$ is dense in $X$. Converse can be proved along the same.\\
	(iv)$\Leftrightarrow$(v) follows by Lemma \ref{7}.
\end{proof}
\begin{corollary}
	Let $(X,f_{1,\infty})$ be an NDDS. If the system is strongly transitive, then every nonempty, $-$ invariant subset $U$ of $X$ is dense in $X$.
\end{corollary}
\begin{proof}
	Let $x\in U$. Since $U$ is $-$ invariant, $ O^-(x,f_{1,\infty})\subset U$. By the condition (iv) of Theorem \ref{10}, $ O^-(x,f_{1,\infty})$ is dense in $X$. Hence, $U$ is dense in $X$.
\end{proof}

\begin{theorem}\label{11}
	For an NDDS $(X,f_{1,\infty})$ the following are equivalent:
	\begin{enumerate}[label=\upshape(\roman*), leftmargin=*, widest=iii]
		\item  The system $(X,f_{1,\infty})$ is strongly extended transitive.
		\item For every opene set $U\subset X$ and every point $x\in X$, there exists $\alpha\in \sum$ such that $x\in f_{\alpha}(U)$.
		\item  For every opene set $U\subset X$ and every point $x\in X$, the set $N_e(U,x)$ is nonempty.
		\item  For every $x\in X$, the extended negative orbit $J^-(x)$ is dense in $X$.
		\item  For every $x\in X$, and $\epsilon>0$, there exists a finite subset $\sum'\subset \sum$ such that $J_{\sum'}^-(x)$ is $\epsilon$ dense in $X$.
		\item  If $U\subset X$ is nonempty and strong $-$ invariant, then $U$ is dense in $X$.
	\end{enumerate}
\end{theorem}
\begin{proof}
	It is straight forward to observe (i)$\Leftrightarrow$(ii) from the definition of strong extended transitivity.\\
	(ii)$\Leftrightarrow$(iii) follows from the definition of $N_e(U,x)$.\\
	(ii)$\Leftrightarrow$(iv): Assume (ii) and suppose $J^-(x)$ is not dense in $X$. Then there exists opene $V\subset X$ such that $\bigcup\limits_{\alpha\in \sum}f_{\alpha}^{-1}(x)\cap V=\emptyset$ which implies that $f_{\alpha}^{-1}(x)\cap V=\emptyset$, for all $\alpha\in \sum$. i.e., $x\notin f_{\alpha}(V)$ for any $\alpha\in \sum$. This contradicts condition (ii). Hence, $J^-(x)$ is dense in $X$. Converse can be proved along the same.\\
	(iv)$\Leftrightarrow$(v) follows by Lemma \ref{7}.\\
	(iv)$\Leftrightarrow$(vi): $J^-(x)$ is strong $-$ invariant and if $x\in U$ and $U$ is strong $-$ invariant then $J^-(x)\subset U$, by Proposition \ref{P2}.
\end{proof}

\begin{theorem}\label{12}
	For an NDDS $(X,f_{1,\infty})$ the following are equivalent:
	\begin{enumerate}[label=\upshape(\roman*), leftmargin=*, widest=iii]
		\item  The system $(X,f_{1,\infty})$ is very strongly transitive.
		\item  For all $\epsilon>0$, there exists $N\in \mathbb{N}$ such that $O_N^-(x,f_{1,\infty})$ is $\epsilon$ dense in $X$ for every $x\in X$.
	\end{enumerate}
\end{theorem}
\begin{proof}
	First we will prove (i)$\Rightarrow$(ii). Let $\epsilon>0$ and consider a cover $\{V_i\}_{i\in \mathbb{I}}$ of $X$ by $\epsilon/2$ balls. Since $X$ is compact, there exists a finite subcover say $\{V_1,V_2,\dots ,V_m\}$. 
	Since $(X,f_{1,\infty})$ is very strong transitive for every $i\in \{1,2,\dots,m\}$, there exists $N_i\in \mathbb{N}$ such that $\bigcup\limits_{n=1}^{N_i}f_1^n(V_i)=X$, for $i=1,2,\dots, m$. Take $N=max\{N_1,N_2,\dots, N_m\}$. Then $\bigcup\limits_{n=1}^{N}f_1^n(V_i)=X$ for $i=1,2,\cdots, m$. Therefore, for every $x\in X$ there exists $i\in \{1, 2,\dots, m\}$ and $m\in \{1, 2, \cdots, N\}$ such that $x\in f_1^n(V_i)$. i.e., $f_1^{-n}(x)\cap V_i \neq \emptyset$. Hence, $O_N^-(x,f_{1,\infty})$ intersects $V_i$ for every $i\in \{1,2,\dots,m\}$. Therefore, $O_N^-(x,f_{1,\infty})$ is dense in $X$.\\  
	For (ii)$\Rightarrow$(i): Let $\epsilon>0$ and let $N\in\mathbb{N}$ be such that $O_N^-(x,f_{1,\infty})$ is $\epsilon$ dense in $X$ for every $x\in X$. Let $V$ be an $\epsilon$ ball in $X$. Then for every $x\in X$, $f_1^{-n}(x)\cap V\neq\emptyset$ for some $n$ with $1\leq n\leq N$. Therefore $\bigcup\limits_{n=1}^N f_1^n (V)=X$. Also, this works for all $\epsilon>0$ and hence is true for every opene $V\subset X$. This implies that the system is very strongly transitive.
\end{proof}
For surjective open maps we see in the next theorem that very strong transitivity is equivalent to strong transitivity.
\begin{theorem}
	Let $(X,f_{1,\infty})$ be an NDDS and $f_n$ is surjective open map for every $n\in \mathbb{N}$. Then the following are equivalent:
	\begin{enumerate}[label=\upshape(\roman*), leftmargin=*, widest=iii]
		\item The system $(X,f_{1,\infty})$ is very strongly transitive.
		\item The system $(X,f_{1,\infty})$ is strongly transitive.
		\item There is no proper, closed, strong $-$ invariant subset of $X$.
	\end{enumerate}
\end{theorem}
\begin{proof}
	(i)$\Leftrightarrow$(ii): We can see that (i)$\Rightarrow$(ii) is true by the definitions.\\
	To prove (ii)$\Rightarrow$(i): Assume that for every opene subset $U\subset X$, $\bigcup\limits_{n=1}^{\infty}f_1^n(U)=X$. Since $f_n$ is an open map for all $n\in \mathbb{N}$, then $f_1^n(U)$ is open for all $n\in \mathbb{N} $. Thus, $\bigcup\limits_{n=1}^{\infty}f_1^n(U)$ is an open cover of $X$ and compactness of $X$ implies that there exists $N\in \mathbb{N}$ such that $\bigcup\limits_{n=1}^{N}f_1^n(U)=X$. Hence, the system is very strongly transitive.\\
	(ii)$\Leftrightarrow$(iii): First we prove (ii)$\Rightarrow$(iii). Suppose $X$ contains a proper, closed and strong $-$ invariant subset $A\subset X$. Then $f_n^{-1}(A)\subset A$ for all $n\in \mathbb{N}$. If $x\in A$, then $J^-(x)\subset A$ by Proposition \ref{P2}. Since $A$ is not dense in $X$, $J^-(x)$ is not dense in $X$ and so, $O^-(x,{f_{1,\infty}})$ is not dense in $X$, which contradicts the condition (ii). Hence, $X$ does not contain a proper, closed, strong $-$ invariant subset.\\
	For (iii)$\Rightarrow$(ii): Suppose $B\subset X$ is nonempty and strong $-$ invariant. $\overline B$ is closed, nonempty and strong $-$ invariant set by condition (iv) of Lemma \ref{6}. From condition (iii), $\overline B=X$ and hence $B$ is dense in $X$. So, the system is strong transitive system by condition (vi) of Theorem \ref{11}.
\end{proof}
\subsection{Extended minimality}

\begin{definition}[Extended minimality] \normalfont
	A system $(X,f_{1,\infty})$ is called \emph{extended minimal} if there is no proper, nonempty, closed strong + invariant subset of $X$.
\end{definition}
In the following theorem we provide equivalent conditions for extended minimality of a system $(X,f_{1,\infty})$. 
\begin{theorem}\label{minimality}
	Let $(X,f_{1,\infty})$ be an NDDS. Then the following are equivalent:
	\begin{enumerate}[label=\upshape(\roman*), leftmargin=*, widest=iii]
		\item  The system $(X,f_{1,\infty})$ is extended minimal.
		\item  For every opene set $U\subset X$ and every point $x\in X$, there exists $\alpha\in \sum$ such that $f_{\alpha}(x)\in U$.
		\item For every opene set $U\subset X$ and every point $x\in X$, the set $N_e(x,U)$ is nonempty.
		\item  For every $x\in X$, the extended orbit $J(x)$ is dense in $X$.
		\item The set $T_e(f_{1,\infty})$ of all extended transitive points is equal to the entire space $X$.
		\item  For every opene set $U\subset X$, $\bigcup\limits_{\alpha\in\sum}f_{\alpha}^{-1}(U)=X$.
		\item  For every opene set $U\subset X$, there exists a finite subset $\sum'\subset \sum$ such that $\bigcup\limits_{\alpha\in\sum'}f_\alpha^{-1}(U)=X$.
		\item  If $A\subset X$ is nonempty, closed and strong + invariant, then $A=X$.
	\end{enumerate}
\end{theorem}
\begin{proof}
	(i)$\Leftrightarrow$(viii) follows from the definition of extended minimality.\\
	(ii)$\Leftrightarrow$(iii) is trivial from the definition of $N_e(x,U)$.\\
	(ii)$\Leftrightarrow$(iv): We prove (ii)$\Rightarrow$(iv) and the proof of converse is similar. Suppose there exists $x\in X$ such that $J(x)$ is not dense in $X$.  i.e., there exists $x\in X$ such that $J(x)\cap U=\emptyset$, for some opene $U\subset X$. This implies that $f_{\alpha}(x)\notin U$, for any $\alpha\in \sum$, which contradicts condition (ii). \\
	(iv)$\Leftrightarrow$(v) follows from the definition of extended orbit and extended transitive point. \\
	(iv)$\Leftrightarrow$(vi): For (vi)$\Rightarrow$(iv) suppose $J(x)$ is not dense in $X$, then there exists an opene subset $B\subset X$ such that $J(x)\cap B=\emptyset$. This implies that $f_{\alpha}(x)\notin B$ for all $\alpha\in \sum $ and so $x\notin f_{\alpha}^{-1}(B)$ for all $\alpha\in\sum$. This is a contradiction to condition (vi). The converse is similar.\\
	(vi)$\Leftrightarrow$(vii): (vii)$\Rightarrow$(vi) is trivial and (vi)$\Rightarrow$(vii) follows from the compectness of $X$.\\
	(iv)$\Leftrightarrow$(viii): First we will prove (iv)$\Rightarrow$(viii). Let $A$ be a nonempty, closed and strong + invariant. If $x\in A$, then $f_{\alpha}(x)\in A$ for all ${\alpha}\in \sum $, which implies $J(x)\subset A$ by Proposition \ref{P1}. Since $J(x)$ is dense in $X$, $A$ is dense in $X$. Also, $A$ is closed hence, $A=\overline A=X$.\\
	To prove (viii)$\Rightarrow$(iv): For any $x\in X$, $J(x)$ is strong + invariant and nonempty by Proposition \ref{P1}. From condition (iii) of Lemma \ref{6}, $\overline {J(x)}$ is strong + invariant. So $\overline {J(x)}$ is nonempty, closed and strong + invariant, and so $\overline {J(x)}=X$ by condition (viii). Hence, $J(x)$ is dense in $X$ for every $x\in X$.
\end{proof}

\subsection{Exact transitivity and strong exact transitivity}
\begin{definition}[Exact and fully exact] \normalfont
	Let $(X,f_{1,\infty})$ be an NDDS. If for every pair of opene subsets $U,V\subset X$, there exists $N\in \mathbb{N}$ such that $f_1^N (U)\cap f_1^N(V)\neq \emptyset$, then the system is called \emph{exact}.The system is called \emph{fully exact} if for all opene $U,V\subset X$, there exists $N\in \mathbb{N}$ such that $(f_1^N (U)\cap f_1^N(V))^{\circ}\neq \emptyset$.
\end{definition}
For open maps, given any opene set $U\subset X,\  f_1^n(U)$ and $(f_1^n(U))^{\circ}$ are the same. Thus, in case of open maps, a system is exact if and only if it is fully exact.\\
Next, we give a result for exact systems and an equivalent condition for a system to be fully exact.

\begin{theorem}
	\begin{enumerate}[label=\upshape(\roman*), leftmargin=*, widest=iii]
		\item If a system $(X,f_{1,\infty})$ is exact and $f_n$ is injective for every $n\in \mathbb{N}$, then the system is trivial, i.e., $X$ is singleton.
		\item The system $(X,f_{1,\infty})$ is fully exact if and only if 
		for every opene subsets $U,V\subset X$ we have $(\bigcup \limits_{n=1}^{\infty}f_1^n (U)\cap f_1^n(V))^{\circ}\neq \emptyset$.
	\end{enumerate}
\end{theorem}
\begin{proof}
	(i) If $X$ is not singleton, then there exists disjoint opene sets $U,V\subset X$. Since $f_n$ is injective for every $n\in \mathbb{N}$, we have $f_1^n (U)\cap f_1^n(V)= \emptyset$ for all $n\in \mathbb{N}$. This contradicts our assumption that $(X,f_{1,\infty})$ is exact. Hence, $X$ must be singleton.\\
	(ii) The necessity part is trivial. To prove the sufficiency part,
	assume that for any two opene subsets $U,V\subset X$ we have  $(\bigcup \limits_{n=1}^{\infty}f_1^n (U)\cap f_1^n(V))^{\circ}\neq \emptyset$. Fix two opene subsets $U,V\subset X$. Let $P\subset U, \ Q\subset V$ be closed sets with nonempty interior. By our assumption the open set $W= (\bigcup\limits _{n=1}^{\infty}f_1^n (P)\cap f_1^n(Q))^{\circ}$ is nonempty, and so it is a Baire space with a countable, relatively closed cover $\{W\cap f_1^n (P)\cap f_1^n(Q):n\in \mathbb{N}\}$. So by the Baire Category Theorem, some $f_1^n (P)\cap f_1^n(Q)$ has nonempty interior in $W$ and so in $X$. Hence $\bigcup \limits_{n=1}^{\infty}(f_1^n (P)\cap f_1^n(Q))^{\circ}\neq \emptyset$ which gives $(\bigcup\limits _{n=1}^{\infty}f_1^n (U)\cap f_1^n(V))^{\circ}\neq \emptyset$. Therefore $(X,f_{1,\infty})$ is fully exact.
\end{proof}
\begin{definition}[Exact transitivity and strong exact transitivity] \normalfont
	A system $(X,f_{1,\infty})$ is called \emph{exact transitive} if for every opene pair $U,V\subset X$, $\bigcup\limits _{n=1}^{\infty}(f_1^n (U)\cap f_1^n(V))$ is dense in $X$, and if $\bigcup \limits_{n=1}^{\infty}(f_1^n (U)\cap f_1^n(V))=X$, then the system is called \emph{strongly exact transitive}.
\end{definition}
Notice that strong exact transitivity implies exact transitivity.
\begin{remark}
	It is easy to see that
	\begin{enumerate}[label=\upshape(\roman*), leftmargin=*, widest=iii]
		\item  if $(X,f_{1,\infty})$ is exact transitive, then it is both exact and topologically transitive,
		\item if $(X,f_{1,\infty})$ is strongly exact transitive, then it is fully exact.
	\end{enumerate}
\end{remark}
Following result gives some characterizations of strongly exact transitive systems.
\begin{theorem}
	Let $(X,f_{1,\infty})$ be an NDDS. Then the following are equivalent:
	\begin{enumerate}[label=\upshape(\roman*), leftmargin=*, widest=iii]
		\item  The system $(X,f_{1,\infty})$ is strongly exact transitive.
		\item  For every pair of opene sets $U,V\subset X$, $\bigcup \limits_{n\in \mathbb{N}}(f_1^n\times f_1^n)(U\times V)$ contains the diagonal $Id_X$.
		\item  For every $x\in X$, the negative orbit $O_{f_{1,\infty}\times f_{1,\infty} }^-(x,x)$ is dense in $X\times X$.  
		\item For every $x\in X$ and opene subsets $U,V \subset X$, there exists $N\in \mathbb{N}$ such that $x\in f_1^N(U)\cap f_1^N(V)$.
	\end{enumerate}
\end{theorem}
\begin{proof}
	To complete the proof we establish that (i), (ii) and (iii) are equivalent to (iv).
	(i)$\Leftrightarrow$(iv): The system $(X,f_{1,\infty})$ is strongly exact transitive.\\
	$\Leftrightarrow$ for every opene pair $U,V\subset X$, $\bigcup \limits_{n=1}^{\infty}(f_1^n(U)\cap f_1^n(V))=X$.\\
	$\Leftrightarrow$ for every $x\in X$ and opene $U,V \subset X$, there exists $N\in \mathbb{N}$ such that $x\in f_1^N(U)\cap f_1^N(V)$.\\
	(ii)$\Leftrightarrow$ (iv): For every pair of opene sets $U,V\subset X$, $\bigcup \limits_{n\in \mathbb{N}}(f_1^n\times f_1^n)(U\times V)$ contains the diagonal $Id_X$.\\ $\Leftrightarrow$ for every $x\in X$ and opene $U,V\subset X$, there exists $N\in \mathbb{N}$ such that $(x,x) \in f_1^N(U)\times f_1^N(V)$.\\
	$\Leftrightarrow$ for every $x\in X$ and opene $U,V\subset X$, there exists $N\in \mathbb{N}$ such that $x \in f_1^N(U)\cap f_1^N(V)$.\\
	(iii)$\Leftrightarrow$ (iv): For every $x\in X$, the negative orbit $O_{f_{1,\infty}\times f_{1,\infty} }^-(x,x)$ is dense in $X\times X$.  \\
	$\Leftrightarrow$ for every $x\in X$ and opene $U,V\subset X$ there exists $N\in \mathbb{N}$ such that $(x,x)\in (f_1^N\times f_1^N) (U,V) = (x,x)\in f_1^N(U)\times f_1^N(V)$. \\
	$\Leftrightarrow$ for every $x\in X$ and opene $U,V \subset X$, there exists $N\in \mathbb{N}$ such that $x\in f_1^N(U)\cap f_1^N(V)$.
\end{proof}
\subsection{Topologically mixing and locally eventually onto}
In this section, we introduce the notion of topologically mixing and locally eventually onto system and then we give some equivalent conditions for these dynamical properties.
\begin{definition}[Topologically mixing] \normalfont
	An NDDS $(X,f_{1,\infty})$ is called \emph{topologically mixing} if for every pair of opene subsets $U,V\subset X$, there exists $k\in \mathbb{N}$ such that $f_1^n(U)\cap V \neq \emptyset$ for all $n\geq k$. 
\end{definition}
\begin{definition}[Locally eventually onto] \normalfont
	If for every opene subset $V\subset X$, there exists $k\in \mathbb{N}$ such that $f_1^k(V)=X$, then the system is called \emph{locally eventually onto}.
\end{definition}
In autonomous discrete dynamical system, locally eventually onto implies topologically mixing but in non-autonomous discrete dynamical system, if $(X,f_{1,\infty})$ is locally eventually onto and each $f_n$ is surjective then the system is topologically mixing, strongly exact transitive and exact transitive.\\

\begin{theorem}
	For an NDDS $(X,f_{1,\infty})$ the following are equivalent:
	\begin{enumerate}[label=\upshape(\roman*), leftmargin=*, widest=iii]
		\item  The system is topologically mixing.
		\item  For every pair of opene subsets $U,V \subset X$, $N(U,V)$ is co-finite.
		\item  For every opene subset $U\subset X$, and $\epsilon>0$, there exists $N\in \mathbb{N}$ such that $f_1^{-n}(U)$ is $\epsilon$ dense in $X$ for all $n\geq N$.
		\item  For every opene subset $U\subset X$, and $\epsilon>0$, there exists $N\in \mathbb{N}$ such that $f_1^{n}(U)$ is $\epsilon$ dense in $X$ for all $n\geq N$.
	\end{enumerate}
\end{theorem}
\begin{proof}
	(i)$\Leftrightarrow$(ii) follows from the definition of topologically mixing.\\
	For (iii)$\Rightarrow$(ii): Let $U,V\subset X$ be opene sets. We choose $\epsilon>0$ such that $U$ contains an $\epsilon$ ball. By condition (iii), there exists $n\in \mathbb{N}$ such that $f_1^{-n}(V)$ is $\epsilon$ dense in $X$ for all $n\geq N$. Then it follows that $U\cap f_1^{-n}(V)\neq\emptyset$ for all $n\geq N$. Hence, $N(U,V)$ is co-finite. (iv)$\Rightarrow$(ii) is similar as (iii)$\Rightarrow$(ii).\\
	Now to prove (ii)$\Rightarrow$(iv): Let $\{U_1,U_2,\dots, U_m\}$ be a finite subcover of $X$ by $\epsilon/2$ balls. Let $V$ be any opene set in $X$, then by condition (ii), $N(V,U_1),N(V,U_2),\dots N(V,U_m)$ are co-finite sets. Clearly, finite intersection of these co-finite sets is again co-finite. So, there exists $N\in \mathbb{N}$ such that for every $i\in \{1,2,\dots, m\}$, $f_1^n(V)\cap U_i \neq \emptyset$ for all $n\geq N$. Since $U_i's$ are $\epsilon/2$ balls covering $X$, $f_1^n(V)$ intersects any $\epsilon$ ball in $X$ for all $n\geq N$. Hence, $f_1^n(V)$ is $\epsilon$ dense in $X$ for all $n\geq N$.\\
	(ii)$\Rightarrow$(iii) follows similarly from the fact that $N(U_1,V),N(U_2,V),\dots, N(U_m,V)$ are co-finite sets.
\end{proof}

\begin{theorem}
	For an NDDS $(X,f_{1,\infty})$ the following are equivalent:
	\begin{enumerate}[label=\upshape(\roman*), leftmargin=*, widest=iii]
		\item  The system is locally eventually onto.
		\item  For all $\epsilon>0$, there exists $N\in\mathbb{N}$ such that $f_1^{-N}(x)$ is $\epsilon$ dense in $X$ for every $x\in X$.
		\item For all $\epsilon>0$, there exists $N\in\mathbb{N}$ such that $f_1^{-n}(x)$ is $\epsilon$ dense in $X$ for every $x\in X$ and every $n\geq N$.
	\end{enumerate}
\end{theorem}
\begin{proof}
	To complete the proof we prove that (i)$\Rightarrow$(iii)$\Rightarrow$(ii)$\Rightarrow$(i).\\
	(i)$\Rightarrow$(iii): Consider a finite subcover $\{V_1,V_2,\dots,V_m\}$ of $X$ by $\epsilon/2$ balls using compactness of $X$. By condition (i), for every $i\in \{1,2,\dots,m\}$, there exists $N_i\in \mathbb{N}$ such that $f_1^n(V_i)=X$ for all $n\geq N_i$. Now take $N=max\{N_1,N_2,\dots,N_m\}$, then for all $i\in \{1,2,\dots,m\}$ and $n\geq N$, $f_1^n(V_i)=X$. This implies that for every $V_i$, $f_1^{-n}(x)\cap V_i\neq\emptyset$ for every $x\in X$ and every $n\geq N$. Since $V_i's$ are $\epsilon/2$ balls covering $X$, condition (iii) follows. (iii)$\Rightarrow$(ii) is trivial.\\
	To prove (ii)$\Rightarrow$(i): Let $V$ be any opene set in $X$. We choose $\epsilon>0$ such that $V$ contains an $\epsilon$ ball. Then by condition (ii), there exists $N\in \mathbb{N}$ such that $f_1^{-N}(x)$ is $\epsilon$ dense in $X$. Then $f_1^{-N}(x)\cap V\neq\emptyset$ for every $x\in X$, i.e., $\{x\}\cap f_1^N(V)\neq \emptyset$ for every $x\in X$. Thus $f_1^N(V)=X$.
\end{proof}
\subsection{Semiconjugacy}\label{conju}
Let $(X,f_{1,\infty})$ and $(Y,g_{1,\infty})$ be two non-autonomous discrete dynamical systems. If $\pi : X \rightarrow Y$ is a continuous surjection such that $\pi \circ f_1^n= g_1^n\circ \pi$ for each $n\in \mathbb{N}$, 
then $\pi$ is called a \emph{semiconjugacy}. Further if $\pi \circ f_n= g_n\circ \pi$ for each $n\in \mathbb{N}$, then $\pi$ is called a \emph{strong semiconjugacy}.

\begin{theorem}\label{20}
	Let $\pi : (X,f_{1,\infty})\rightarrow  (Y,g_{1,\infty})$ be a semiconjugacy. If $(X,f_{1,\infty})$ is topologically transitive, strongly transitive, very strongly transitive, exact transitive, strongly exact transitive, exact or locally eventually onto, then $(Y,g_{1,\infty})$ satisfies the corresponding property.
\end{theorem}
\begin{proof}
	Suppose $(X,f_{1,\infty})$ is strongly transitive. Then for any opene subset $U$ in $Y$, $\pi^{-1}(U)$ is opene in $X$ and so $X=\bigcup\limits_{n=1}^{\infty}f_1^n(\pi^{-1}(U))$. It follows that $Y=\pi(X)=\bigcup\limits_{n=1}^{\infty}\pi f_1^n(\pi^{-1}(U))=\bigcup\limits_{n=1}^{\infty}g_1^n\pi (\pi^{-1}(U))=\bigcup\limits_{n=1}^{\infty}g_1^n(U)$. Thus $(Y,g_{1,\infty})$ is strongly transitive.\\
	If $(X,f_{1,\infty})$ is very strongly transitive then we use the same proof with $\bigcup\limits_{n=1}^{\infty}f_1^n(\pi^{-1}(U))$ is replaced by $\bigcup\limits_{n=1}^{N}f_1^n(\pi ^{-1}(U))$ for sufficiently large $N$ depending on $\pi^{-1}(U)$ and obtain the result with $\bigcup\limits_{n=1}^{\infty}g_1^n(U)$ replaced by $\bigcup\limits_{n=1}^{N}g_1^n(U)$. If $(X,f_{1,\infty})$ is locally eventually onto, we replace $\bigcup\limits_{n=1}^{\infty}f_1^n(\pi ^{-1}(U))$ by $f_1^N(\pi ^{-1}(U))$ and obtain the result with $\bigcup\limits_{n=1}^{\infty}g_1^n(U)$ replaced by $g_1^N(U)$. The remaining properties can be proved similarly by using the inclusion $\pi[f_1^n(\pi ^{-1}(U)) \cap f_1^n(\pi ^{-1}(V)) ]\subset g_1^n(U)\cap g_1^n(V)$.
\end{proof}
Proof of the following result is similar to the proof of Theorem \ref{20}.
\begin{theorem}\label{21}
	Let $\pi : (X,f_{1,\infty})\rightarrow  (Y,g_{1,\infty})$ be a strong semiconjugacy. If $(X,f_{1,\infty})$ is extended transitive or extended strongly transitive, then $(Y,g_{1,\infty})$ satisfies the corresponding property.
\end{theorem}

Given two sequences $f_{1,\infty}=(f_n)_{n\in \mathbb{N}}$ and $g_{1,\infty}=(g_n)_{n\in \mathbb{N}}$ of self maps on $X$ and $Y$ respectively, we denote the sequence $(f_n \times g_n)_{n\in \mathbb{N}}$ of self maps on $X\times Y$ by $f_{1, \infty }\times g_{1,\infty}$.\\
Using the fact that every projection is a semiconjugacy, we can easily prove the following theorem.
\begin{theorem} 
	If $(X\times Y, f_{1,\infty}\times g_{1,\infty})$ is strongly transitive, very strongly transitive, exact transitive, strongly exact transitive, exact or locally eventually onto, then both $(X,f_{1,\infty})$ and $(Y,g_{1,\infty})$ satisfy the corresponding property.
\end{theorem}

\begin{theorem}
	Assume $(Y,g_{1,\infty})$ is topologically mixing. 
	If $(X,f_{1,\infty})$ is topologically transitive or topologically mixing, then $(X\times Y, f_{1,\infty}\times g_{1,\infty})$ satisfies the corresponding property.
\end{theorem}
\begin{proof}
	Let $U_1, V_1\subset X$ and $U,V\subset Y$ be opene subsets. Since $(Y,g_{1,\infty})$ is topologically mixing, $N(U,V)$ is co-finite. Hence $N(U_1,V_1)\cap N(U,V)$ is infinite whenever $N(U_1,V_1)$ is infinite, therefore $(X\times Y, f_{1,\infty}\times g_{1,\infty})$ is topologically transitive whenever $(X,f_{1,\infty})$ is topologically transitive. Similarly, $N(U_1,V_1)\cap N(U,V)$ is co-finite whenever $N(U_1,V_1)$ is co-finite, therefore $(X\times Y, f_{1,\infty}\times g_{1,\infty})$ is topologically mixing whenever $(X,f_{1,\infty})$ is topologically mixing.	
\end{proof}
\begin{theorem}
	Assume $(Y,g_{1,\infty})$ is locally eventually onto. If $(X,f_{1,\infty})$ is strongly transitive, very strongly transitive, exact, fully exact, exact transitive, strongly exact transitive or locally eventually onto, then $(X\times Y, f_{1,\infty}\times g_{1,\infty})$ satisfies the corresponding property.	
\end{theorem}
\begin{proof}
	Let $U_1 \subset X$ and $U\subset Y$ be opene subsets. Since $(Y,g_{1,\infty})$ is locally eventually onto, there exists $N\in \mathbb{N}$ such that $g_1^n(U)=Y$ for all $n\geq N$. If $L\subset \mathbb{N}$ so that $\bigcup \{f_1^n(U_1):n\in L\}=X$ and there exists $N\in L $ with $n\geq N$ then $\bigcup \{(f_1^n\times g_1^n)(U_1\times U):n\in L\}=X\times Y$. Using $L=\mathbb{N}$ we obtain the result for strong transitivity. Using $L=\{1,\dots N_1\}$ with $N_1$ sufficiently large, we obtain the result for very strong transitivity. Using $L=\{N_1\}$ with $N_1$ sufficiently large, we obtain the result for locally eventually onto. Remaining properties can be proved easily.
\end{proof}

In the next three theorems, we provide the sufficient conditions under which the topological transitivity and mixing are preserved by every finite rearrangement and the locally eventually onto property is preserved by every rearrangement of the given system.
 
\begin{proposition}\label{rearrangement1}
	Let $(X,f_{1,\infty})$ be an NDDS where $f_m\circ f_n=f_n\circ f_m$ for all $m, n\in \mathbb{N}$. If $(X,f_{1,\infty})$ is topologically transitive then $(X,g_{1,\infty})$ is toplogically trasnitive for every finite rearrangement $g_{1,\infty}$ of $f_{1,\infty}$.
\end{proposition}
\begin{proof}
Let $g_{1,\infty}$ is a finite rearrangement of $f_{1,\infty}$, i.e., there exists $N\in \mathbb{N}$ such that $\{f_1, f_2,\dots, f_N\}=\{g_1, g_2,\dots, g_N\}$ and $f_i=g_i$ for all $i\geq N+1$. Let $U, V\subset X$ be opene. Since  $f_m\circ f_n=f_n\circ f_m$ for all $m, n\in \mathbb{N}$, we can see that $f_1^k(U) = g_1^k(U)$ for all $k\geq N$. Since $(X,f_{1,\infty})$ is topologically transitive, there exists $k\geq N$ such that $f_1^k(U)\cap V\neq \emptyset$. Therefore $g_1^k(U)\cap V\neq \emptyset$. Hence, $(X,g_{1,\infty})$ is topologically transitive.
	\end{proof}

\begin{proposition}\label{rearrangement2}
	Let $(X,f_{1,\infty})$ be an NDDS where $f_m\circ f_n=f_n\circ f_m$ for all $m, n\in \mathbb{N}$. If $(X,f_{1,\infty})$ is topologically mixing then $(X,g_{1,\infty})$ is toplogically mixing for every finite rearrangement $g_{1,\infty}$ of $f_{1,\infty}$. 
\end{proposition}
\begin{proof}
	The proof is similar to that of Theorem \ref{rearrangement1}.	
\end{proof}
\begin{proposition}\label{rearrangement3}
	Let $(X,f_{1,\infty})$ be an NDDS where each $f_n$ is surjective and $f_m\circ f_n=f_n\circ f_m$ for all $m, n\in \mathbb{N}$. If $(X,f_{1,\infty})$ is locally eventually onto then $(X,g_{1,\infty})$ is locally eventually onto for every rearrangement $g_{1,\infty}$ of $f_{1,\infty}$.
\end{proposition}
\begin{proof}
	Suppose $(X,f_{1,\infty})$ is locally eventually onto. Let $V\subset X$ be  opene. Then there exists $k\in \mathbb{N}$ such that $f_1^k(V)=X$. Since $g_{1,\infty}$ is a rearrangement of $f_{1,\infty}$ we can find a natural number $N$ such that $\{f_1, f_2, \dots, f_k\}\subset \{g_1, g_2, \dots, g_N\}$. Then $g_1^N(V)=X$, because each $f_n$ is surjective and $f_m\circ f_n=f_n\circ f_m$ for all $m, n \in \mathbb{N}$. Hence, $(X,g_{1,\infty})$ is locally eventually onto. 
\end{proof}

\section{Generic Dynamical Systems}
We start this section by defining generic dynamical system (GDS), some basic definitions in GDS, various notions of transitivity and at the end of this section we define conjugacy and strong conjugacy of GDS.\\
A \emph{generic dynamical system} is a pair $(X,\mathcal{F})$, where $X$ is a compact metric space and $\mathcal F$ is a collection of continuous self maps on $X$.

For $x\in X, A,B\subset X$ define
\begin{center}
	$  \mathcal{F}(x)=\{f(x):f \in \mathcal{F}\},$\\
	$\displaystyle\mathcal{F}(A)=\bigcup_{f \in \mathcal{F}}f(A)=\bigcup_{x\in A}\mathcal{F}(x),$\\
	$\displaystyle\mathcal{F}^{-1}(B)=\bigcup_{f \in \mathcal{F}} (f)^{-1}(B)$.\\
	
\end{center}

Since $f \in \mathcal{F}$ is continuous, $B$ open implies $\mathcal{F}^{-1}(B)$ is open. We write $\mathcal{F}^{-1}(x)$ for $\mathcal{F}^{-1}(\{x\})$.\\
Observe that equation (\ref{2.1}) implies
\begin{equation}
	\label{4.1}	\mathcal{F}(A)\cap B\neq \emptyset \Leftrightarrow A\cap \mathcal{F}^{-1}(B)\neq \emptyset.
\end{equation}
In particular,
\begin{equation}
	\label{4.2}	\mathcal{F}(x)\cap B\neq \emptyset \Leftrightarrow x\in \mathcal{F}^{-1}
	(B).
\end{equation}

The \emph{orbit} of an element $x\in X$ is defined as
\begin{center}
	$O_{\mathcal{F}}(x)=\{f_n\circ f_{n-1}\circ\cdots\circ f_1(x): f_1, f_2, \dots, f_n \in \mathcal{F}\}$
\end{center} and the \emph{negative orbit} is defined as
\begin{center}
	$O_{\mathcal{F}}^-(x)=\{y\in X: f_n\circ f_{n-1}\circ\cdots\circ f_1(y)=x$ for some $f_1, f_2, \dots, f_n \in \mathcal{F}\}$.
\end{center}

It can be easily seen that $\mathcal{F} (x)\subset O_\mathcal{F}(x)$ and $\mathcal{F}^{-1}(x)\subset O^-_\mathcal{F}(x)$ for every $x\in X$. If $\mathcal{F}$ is a semigroup, then $\mathcal{F} (x)= O_\mathcal{F}(x)$ and $\mathcal{F}^{-1}(x)= O^-_\mathcal{F}(x)$ for every $x\in X$.

A point $x\in X$ is called \emph{transitive point} if the orbit of $x$ is dense in $X$.  The set of all transitive points of $X$ is denoted by $Trans (\mathcal{F})$. The \emph{$\omega$-limit set} of a point $x\in X$, denoted by $\omega(x,\mathcal{F})$ is the set of all limit points of the orbit of $x$.

A subset $A\subset X$ is called \emph{ + invariant} if $\mathcal{F}(A)\subset A$. A subset $A\subset X$ is called \emph{$-$ invariant} if $\mathcal{F}^{-1}(A)\subset A$.

From equation (\ref{2.2}), it follows that $A$ is + invariant if and only if its complement $A^c$ is $-$ invariant. Also, continuity and compactness imply that if $A$ is + invariant, then its closure $\overline{A}$ is + invariant. Clearly, $A$ is + invariant if and only if $\mathcal{F}(x)\subset A$ for all $x\in A$ and is $-$ invariant if and only if $\mathcal{F}^{-1}(x)\subset A$ for all $x\in A$. If each $f$ is surjective, then from equation (\ref{2.3}), it follows that $A$ $ -$ invariant implies $A\subset f(A)$ for all $f \in \mathcal{F}$.

Call \emph{$A$ invariant} when $f(A)= A$ for all $f \in \mathcal{F}$. Call $A$ \emph{weakly  invariant} when $\mathcal{F}(A) $ is a dense subset of $A$. Thus, $A$ is closed and weakly invariant if and only if $ A=\overline{\mathcal{F}(A)}$. Again continuity and compactness imply that $A$ invariant (or weakly invariant) implies $\overline{A}$ is invariant (respectively weakly invariant).\\
It is easy to see that in a generic dynamical system every orbit is + invariant.

If $\mathcal F$ is a semigroup, it is easy to see that $\mathcal{F}(A)$ is + invariant and $\mathcal{F}^{-1}(A)$ is $-$ invariant. In particular, $\mathcal{F}(x)$ and its closure are + invariant. Also $\textit{Trans}(\mathcal{F})$ is $-$ invariant. 
\begin{definition} \normalfont
	
	Given a GDS $(X,\mathcal{F})$, a subset $A\subset X$ is said to be \emph{$\mathcal{F}-$transitive} if for all $U,V$ opene in $X$ there exist $f_1,f_2,\dots, f_n\in \mathcal{F}$ such that $f_n\circ f_{n-1}\circ \cdots f_1(U\cap A)\cap V \ne\emptyset$.
\end{definition}

\begin{proposition}
	Let $(X,\mathcal{F})$ be a GDS where $\mathcal{F}$ is a semigroup and $A\subset X$. Then the following are equivalent:
	\begin{enumerate}[label=\upshape(\roman*), leftmargin=*, widest=iii]
		\item $A$ is $\mathcal{F}-$ transitive.
		\item For all $U,V$ open in $X$ with $U\cap A\ne\emptyset $ and $V\cap A\ne\emptyset $ there exists $f\in \mathcal{F}$ such that $f(A\cap U) \cap V\neq \emptyset$.
		\item $A\cap U \cap \mathcal{F}^{-1}(V)\neq \emptyset$ for all $U,V$ open in $X$ with $U\cap A\ne\emptyset $ and $V\cap A\ne\emptyset $. 
		\item $\mathcal{F}(A\cap U) \cap V\neq \emptyset$ for all $U,V$ open in $X$ with $U\cap A\ne\emptyset $ and $V\cap A\ne\emptyset $. 
		\item $A\cap \mathcal{F}^{-1}(V)$ is dense in $A$ for all $V$ open in $X$ with $V\cap A \neq\emptyset$.
	\end{enumerate}
\end{proposition}
\begin{proof}
	It is easy to see that (ii)$-$(v) are equivalent to each other. (ii)$\Rightarrow$(i) is also trivial. (i)$\Rightarrow$(ii) holds when $\mathcal{F}$ is a semigroup.
\end{proof}	
\begin{theorem}
	For $A\subset X$, the following are equivalent:
	\begin{enumerate}[label=\upshape(\roman*), leftmargin=*, widest=iii]
		\item For all opene $U,V$ in $X$ with $U\cap A\ne\emptyset $ and $V\cap A\ne\emptyset $ there exists $f\in \mathcal{F}$ such that $f(A\cap U) \cap V\neq \emptyset$.\label{1}
		\item For all opene $U,V$ in $X$ with $U\cap \overline{A}\ne\emptyset $ and $V\cap \overline{A}\ne\emptyset $ there exists $f\in \mathcal{F}$ such that $f(\overline A\cap U) \cap V\neq \emptyset$.
		\item $\{x\in \overline{A}: {A}\subset \overline{\mathcal{F}(x)}\}$ is a dense subset of $\overline{A}$.
	\end{enumerate}
	Moreover, if $A$ is a $G_\delta$ subset of $X $ satisfying condition \ref{1}, then $\{x\in {A}: {A}\subset \overline{\mathcal{F}(x)}\}$ is a dense, $G_\delta $ subset of $A$.
\end{theorem}

\begin{proof}
	(i)$\Leftrightarrow$(ii): The open sets, $U, V$ and $U\cap \mathcal{F}^{-1}(V)$ meet the closure of $\overline{A }$ if and only if they meet $A$.\\
	(iii)$\Rightarrow$(ii): Let $B=\{x\in {A}: {A}\subset \overline{\mathcal{F}(x)}\}$. For $x\in A$, $A\subset \overline{\mathcal{F}(x)}$ if and only if $\mathcal{F}(x)\cap V\neq \emptyset $ for every open $V$ such that $A\cap V\neq \emptyset$.\\
	By equation (\ref{4.2}) this occurs if and only if $x\in \mathcal{F}^{-1}(V)$ for all open $V$ such that $A\cap V\neq \emptyset$. That is,
	\begin{equation}
		\label{4.3} B= \bigcap \{\mathcal{F}^{-1}(V): \textit{ $V$ is open and $V\cap A\neq \emptyset$}\}.
	\end{equation}
	Notice that we obtain the same intersection if we restrict to those $V$'s is a countable basis and so $B$ is a $G_\delta $ subset of A.\\
	Thus, if $\{x\in {A}: {A}\subset \overline{\mathcal{F}(x)}\}$ is dense in $A$, then $A\cap \mathcal{F}^{-1}(V)$ is dense in $A$ for all open $V$ such that $A\cap V\neq \emptyset$ and so there exists $f\in \mathcal{F}$ such that $f(A\cap U) \cap V\neq \emptyset$. Applied to the closure of $A$ this yields (iii)$\Rightarrow$(ii).\\
	(ii)$\Rightarrow$(iii): Now suppose that for all opene $U,V$in $X$ there exists $f\in \mathcal{F}$ such that $f(A\cap U) \cap V\neq \emptyset$, and $A$ is a $G_\delta$ subset of $X$ and so is a dense, $G_\delta$ subset of $\overline{A}$. For each $V$ open and such that $A\cap V\neq \emptyset,\ \mathcal{F}^{-1}(V)$ is open and has a dense intersection with $\overline{A}$. Letting $V$ vary over members of a countable base, (\ref{4.3}) implies that $\{x\in \overline{A}: {A}\subset \overline{\mathcal{F}(x)}\}=\{x\in \overline{A}: \overline{A}\subset \overline{\mathcal{F}(x)}\}$ is a dense $G_\delta$ subset of $\overline{A}$ by the Baire Category Theorem (proving (iii)).
	
	Finally if $A$ is $G_\delta$ in $X$ satisfying condition (i), then  
	\begin{center}
		
		$\{x\in {A}: {A}\subset \overline{\mathcal{F}(x)}\}=A\cap \{x\in \overline{A}: \overline{A}\subset \overline{\mathcal{F}(x)}\}$
	\end{center}  is a dense, $G_\delta$ subset of $\overline{A}$ and hence of $A$.
\end{proof}

\begin{definition}[Topologically transitive] \normalfont
	A generic dynamical system $(X,\mathcal{F})$ is called \emph{topologically transitive} if for all opene pair $U,V\subset X$, there exist $f_1, f_2, \dots, f_n \in \mathcal{F}$ such that $f_n\circ f_{n-1}\circ \cdots \circ f_1(U)\cap V\neq \emptyset$.
	
\end{definition}
\begin{proposition}
	Let $(X,\mathcal{F})$ be a GDS where $\mathcal{F}$ is a semigroup and $A\subset X$. Then the following are equivalent:	
	\begin{enumerate}[label=\upshape(\roman*), leftmargin=*, widest=iii]
		\item The system $(X,\mathcal{F})$ is topologically transitive.
		\item $X$ is an $\mathcal{F}-$transitive set.
		\item $\mathcal{F}(U)$ is dense in $X$ for all opene $U$ in $X$.
		\item $\mathcal{F}^{-1}(V)$ is dense in $X$ for all opene $V$ in $X$.
	\end{enumerate}
	\begin{proof}
		It is easy to see that (ii), (iii) and (iv) are equivalent to each other and (ii)$\Rightarrow$(i) is also trivial. (i)$\Rightarrow$(ii) is true if $\mathcal{F}$ is a semigroup.
	\end{proof}
\end{proposition}
\begin{theorem}
	In the semigroup case if every opene $-$ invariant set is dense, then the system is transitive.
	
\end{theorem}
\begin{proof}
	If $V$ is opene, then $\mathcal{F}^{-1}(V)$ is $-$ invariant and so is dense if it is nonempty. Once we show that $\mathcal{F}^{-1}(V)$ is nonempty for every opene $V$, we have $\mathcal{F}^{-1}(V)$ is dense for all such $V$ and this is transitivity.
	
	Assume instead that $\mathcal{F}^{-1}(V)=\emptyset $ for some opene $V$. We obtain a contradiction. Let $x\in V$. Observe that $V$ is opene and $-$ invariant and so is dense.
	
	Case 1. $V=\{x\}$. Since $V=\{x\}$ is dense and closed, it follows that $X=\{x\}$. Then for all $f\in \mathcal F, \ f(x)=x$ and so $\{x\}=\mathcal{F}^{-1}(x)=\mathcal{F}^{-1}(V)$ which is nonempty.\\
	Case 2. There exists $y\neq x$ with $y\in V$. Choose open disjoint subsets $U_x$ and $U_y$ of $V$ with $x\in U_x$  and $y\in U_y$. Hence, $\mathcal{F}^{-1}(U_x)\cup \mathcal{F}^{-1}(U_y)\subset\mathcal{F}^{-1}(V)=\emptyset$. Hence, $U_x$ and $U_y$ are each $-$ invariant and so are open and dense. But then $U_x\cap U_y$ is open and dense which is impossible since $U_x\cap U_y=\emptyset$.
\end{proof}	

In sections 3.1 and 3.2, we gave some equivalent conditions for extended transitivity and strong extended transitivity in NDDS. The following theorem gives such results for topological transitivity in GDS.
\begin{theorem}
	Let $(X,\mathcal{F})$ be a generic dynamical system, where $X$ is a compact perfect metric space and each $f_n$ is surjective. Then the following are equivalent:
	\begin{enumerate}[label=\upshape(\roman*), leftmargin=*, widest=iii]
		\item The system $(X,\mathcal{F})$ is topologically transitive.
		\item For every pair of opene sets $U,V\subset X$, there exist
		$f_1, f_2,\dots, f_n\in \mathcal{F}$ such that $f_1^{-1}\circ\cdots \circ f_n^{-1}(U)\cap V\neq\emptyset$.
		\item For every opene set $U \subset X$, the set $\bigcup \limits_{\substack{n\in \mathbb{N}\\f_1, f_2, \dots, f_n \in \mathcal{F}}}f_n\circ f_{n-1}\circ \cdots \circ f_1(U)$ is dense in $X$.
		\item For every opene set $U \subset X$, the set $\bigcup \limits_{\substack{n\in \mathbb{N}\\f_1, f_2, \dots, f_n \in \mathcal{F}}}f_1^{-1}\circ f_{2}^{-1}\circ \cdots \circ f_n^{-1}(U)$ is dense in $X$.
		\item If $A\subset X$ is closed and + invariant, then either $A=X$ or $A$ is nowhere dense in $X$.
		\item If $U\subset X$ is open and $U\subset f(U)$ for all $f \in \mathcal{F}$, then either $U=\emptyset$ or $U$ is dense in $X$.
		\item There exists $x\in X$ such that the orbit $O_{\mathcal{F}}(x)$ is dense in $X$.
		\item The set $Trans(\mathcal{F})$ is transitive points equals $\{x:\omega(x,\mathcal{F})=X\}$ and it is a dense $G_{\delta}$ subset of $X$.
	\end{enumerate} 
\end{theorem}
\begin{proof}
	The proof is similar to that of Theorem \ref{8}.
\end{proof}

\begin{definition}[Strongly transitive] \normalfont
	A generic dynamical system $(X,\mathcal{F})$ is called \emph{strongly transitive}  if  $$\bigcup \limits_{\substack{n\in \mathbb{N}\\f_1, f_2, \dots, f_n \in \mathcal{F}}}f_n\circ f_{n-1}\circ \cdots \circ f_1(U)=X$$ for every opene subset $U\subset X$.
\end{definition}
If $\mathcal{F}$ is a semigroup then it is easy to see that the GDS $(X,\mathcal{F})$ is strongly transitive if and only if $\mathcal{F}(U)=X$ for all opene $U$ in $X$ if and only if $\mathcal{F}^{-1}(x)$ is dense for all $x\in X$.\\
By the similar technique used in Theorem \ref{11}, we get 
\begin{theorem}
	For a generic dynamical system $(X,\mathcal{F})$ the following are equivalent:
	\begin{enumerate}[label=\upshape(\roman*), leftmargin=*, widest=iii]
		\item The system $(X, \mathcal{F})$ is strongly transitive.
		\item For every opene set $U \subset X$ and every point $x\in X$, there exist
		$f_1, f_2,\dots, f_n\in \mathcal{F}$ such that $x \in f_n\circ\cdots \circ f_1(U)$.
		\item The negative orbit of $x$ is dense in $X$ for every $x\in X$.
	\end{enumerate}
\end{theorem}
\begin{definition}[Very strongly transitive] \normalfont
	Let $(X,\mathcal{F})$ be a generic dynamical system where $\mathcal{F}$ is a topological space. We say that $(X,\mathcal{F})$ is \emph{very strongly transitive}  if  for every opene subset $U\subset X$, there exists a compact subset $S\subset \mathcal{F}$ such that $\bigcup \limits_{f\in S} f(U)=X$.
\end{definition}
From the definition it immediately follows that\\
Very strong transitivity $\Rightarrow$ strong transitivity $\Rightarrow$ topological transitivity.

\begin{definition}[Minimality in GDS] \normalfont
	A system $(X, \mathcal{F})$ is called \emph{minimal} if there is no proper, nonempty, closed, + invariant subset of $X$.
\end{definition}
In section 3.3 we have given some equivalent conditions for extended minimality. Now we give similar equivalent conditions for minimality in the case of GDS. The main idea of the proof of Theorem \ref{minimal} is similar to that of Theorem \ref{minimality}. 
\begin{theorem}\label{minimal}
	Let $(X, \mathcal{F})$ be a GDS. Then the following are equivalent:
	\begin{enumerate}[label=\upshape(\roman*), leftmargin=*, widest=iii]
		\item  The system $(X, \mathcal{F})$ is minimal.
		\item  For every opene set $U\subset X$ and every point $x\in X$, there exists $f_1, f_2,\dots, f_n\in \mathcal{F}$ such that $f_n\circ f_{n-1}\circ \cdots \circ f_2\circ f_1(x)\in U$.
		\item  For every $x\in X$, the orbit $O_{\mathcal{F}}(x)$ is dense in $X$.
		\item The set $Trans(\mathcal{F})$ of all transitive points is equal to the entire space $X$.
		\item  For every opene set $U\subset X$, $\bigcup\limits_{\substack{n\in \mathbb{N}\\f_1, f_2, \dots, f_n \in \mathcal{F}}}f_1^{-1}\circ f_2^{-1}\circ\cdots \circ f_n^{-1}(U)=X$.
		\item  For every opene set $U\subset X$, there exists a finite subset $\mathcal{G}\subset \mathcal{F}$ such that $\bigcup\limits_{\substack{n\in \mathbb{N}\\g_1, g_2, \dots, g_n \in \mathcal{G}}}g_1^{-1}\circ g_2^{-1}\circ\cdots\circ g_n^{-1}(U)=X$.
		\item  If $A\subset X$ is nonempty, closed and + invariant, then $A=X$.
	\end{enumerate}
\end{theorem}
\begin{definition}[Topologically mixing] \normalfont
	Let $(X,\mathcal{F})$ be a generic dynamical system where $\mathcal{F}$ is a topological space.	We say that $(X,\mathcal{F})$ is \emph{topologically mixing}  if  for every opene pair $U, V\subset X$, there exists a compact subset $S\subset \mathcal{F}$ such that $f(U)\cap V\neq\emptyset$ for all $f\in \mathcal{F}\setminus S$.
\end{definition}
\begin{definition}[Locally eventually onto] \normalfont
	Let $(X,\mathcal{F})$ be a generic dynamical system where $\mathcal{F}$ is a topological space.	We say that $(X,\mathcal{F})$ is \emph{locally eventually onto} if  for every opene subset $U\subset X$, there exists a compact subset $S\subset \mathcal{F}$ such that $f(U)=X$ for all $f\in \mathcal{F}\setminus S$.
\end{definition}
From the definitions, we obtain that 
locally eventually onto implies topologically mixing.
\begin{definition}[Conjugacy of generic dynamical system] \normalfont
	Let $(X,\mathcal{F})$ and $(Y,\mathcal{G})$ be two generic dynamical systems. A \emph{semiconjugacy} from $(X,\mathcal{F})$ to $(Y,\mathcal{G})$ is a pair $(\phi, h)$ where $\phi : X\rightarrow Y$ is a continuous surjective map and $h: \mathcal{F}\rightarrow \mathcal{G}$ is a surjective map such that $\phi (f(x))=h(f)(\phi (x))$ for all $x\in X,\  f\in \mathcal{F}$. If $\phi$ is a homeomorphism and $h$ is a bijection then $(\phi,h)$ is called a \emph{conjugacy}  from $(X,\mathcal{F})$ to $(Y,\mathcal{G})$.
\end{definition}
If $(\phi, h)$ is a semiconjugacy from $(X,\mathcal{F})$ to $(Y,\mathcal{G})$ and $\mathcal{F}$ is group, then it is easy to see that $\mathcal{G}$ is also a group and $h$ is a group homomorphism. Further, if $\phi $ is homeomorphism, then $(\phi, \tilde{h})$ is a conjugacy from $(X,\mathcal{\tilde{F}})$ to $(Y,\mathcal{G})$ where $\mathcal{\tilde{F}}$ is a quotient group $\displaystyle\frac{\mathcal{F}}{\ker h}$ and $\tilde{h}([f])=h(f)$ for every $f\in \mathcal{F}$. 

\begin{definition}[Strong conjugacy of generic dynamical system] \normalfont
	Let $(X,\mathcal{F})$ and $(Y,\mathcal{G})$ be two generic dynamical systems where $\mathcal{F}$ and $\mathcal{G}$ are topological spaces. A \emph{strong semiconjugacy} from $(X,\mathcal{F})$ to $(Y,\mathcal{G})$ is a pair $(\phi,h) $ where $\phi : X\rightarrow Y$ and $h: \mathcal{F}\rightarrow \mathcal{G}$ are continuous surjective maps such that $\phi (f(x))=h(f)(\phi (x))$ for all $x\in X, \ f\in \mathcal{F}$. If $\phi, h$ are homeomorphisms then $(\phi, h)$ is \emph{strong conjugacy}.
\end{definition}
The next result says that, in the case of GDS, topological transitivity and strong transitivity are preserved under semiconjugacy. 
\begin{theorem}\label{conjugacy}
	Let $(X,\mathcal{F})$ and $(Y,\mathcal{G})$ be two generic dynamical systems. Suppose there exists a semiconjugacy from $(X,\mathcal{F})$ to  $(Y,\mathcal{G})$. If $(X,\mathcal{F})$ is either topologically transitive and strongly transitive, then $(Y,\mathcal{G})$ satisfies the corresponding property.
\end{theorem}
\begin{proof}
	The proof is similar to that of Theorem \ref{20}.
\end{proof}
The next theorem says that, in a GDS, very strong transitivity, topological mixing and locally eventually onto are preserved under strong semiconjugacy. 
\begin{theorem}\label{strongconjugacy}
	Let $(X,\mathcal{F})$ and $(Y,\mathcal{G})$ be two generic dynamical systems where $\mathcal{F},\mathcal{G}$ are topological spaces. Suppose there exists a strong semiconjugacy from $\phi : (X,\mathcal{F})$ to $(Y,\mathcal{G})$. If $(X,\mathcal{F})$ is very strongly transitive, topologically mixing or locally eventually onto, then $(Y,\mathcal{G})$ satisfies the corresponding property.
\end{theorem}
\begin{proof}
	The proof is similar to that of Theorem \ref{20}.
\end{proof}
We now associate a GDS $(X,\mathcal{F})$ corresponding to the given NDDS $(X,f_{1,\infty})$ and investigate whether if the given NDDS $(X,f_{1,\infty})$ has a particular variation of transitivity then the associated GDS $(X,\mathcal{F})$ also has such a variation of transitivity and vice versa. The following three theorems deal with this problem. These results are easy to prove. So we state them without proof.
\begin{theorem}\label{4.19}
	Let $(X,f_{1,\infty})$ be a non-autonomous discrete dynamical system. Let $\mathcal{F}=\{f_n: n\in \mathbb{N}\}$. Then 
	\begin{enumerate}[label=\upshape(\roman*), leftmargin=*, widest=iii]
		\item the NDDS $(X,f_{1,\infty})$ is extended transitve if and only if the generic dynamical system $(X,\mathcal{F})$ is topologically transitive, 
		\item the NDDS $(X,f_{1,\infty})$ is strongly extended transitve if and only if the generic dynamical system $(X,\mathcal{F})$ is strongly transitive,
		\item the NDDS $(X,f_{1,\infty})$ is extended minimal if and only if the generic dynamical system $(X,\mathcal{F})$ is minimal. 
	\end{enumerate}
\end{theorem}
\begin{theorem}\label{4.20}
	Let $(X,f_{1,\infty})$ be a non-autonomous discrete dynamical system. Consider discrete topology on the set $\mathcal{F}=\{f_1^n : n\in \mathbb{N}\}$. Then 
	\begin{enumerate}[label=\upshape(\roman*), leftmargin=*, widest=iii]
		\item the NDDS $(X,f_{1,\infty})$ is very strongly transitive if and only if the generic dynamical system $(X,\mathcal{F})$ is very strongly transitive, 
		\item the NDDS $(X,f_{1,\infty})$ is topologically mixing if and only if the generic dynamical system $(X,\mathcal{F})$ is topologically mixing.
	\end{enumerate}
\end{theorem}
\begin{theorem}\label{4.21}
	Let $(X,f_{1,\infty})$ be a non-autonomous discrete dynamical system where each $f_n$ is surjective. Consider discrete topology on the set $\mathcal{F}=\{f_1^n : n\in \mathbb{N}\}$. Then the NDDS $(X,f_{1,\infty})$ is locally eventually onto if and only if the generic dynamical system $(X,\mathcal{F})$ is locally eventually onto.
\end{theorem}
The next two theorems are straight forward and we omit the details.
\begin{theorem}
	Let $(X, \mathcal{F})$ be a GDS and $\mathcal{G}$ be the semigroup generated by $\mathcal{F}$. If $(X,\mathcal{F})$ is topologically transitiive, strongly transitive or minimal then $(X,\mathcal{G})$ satisfies the corresponding property.
\end{theorem}

\begin{theorem}
	Let $(X, f_{1,\infty})$ be an NDDS and $\mathcal{F}=\{f_{\alpha}: \alpha\in \sum\}$. If the NDDS $(X, f_{1,\infty})$ is topologically transitiive, strongly transitive or extended minimal then $(X,\mathcal{F})$ satisfies the corresponding property.
\end{theorem}

\section{Conclusions}
We studied the variations of topological transitiviy in NDDS and GDS. In the case of NDDS, semiconjugacy preserves topological transitivity, strong transitivity, very strong transitivity, exact transitivity, strong exact transitivity, exact and locally eventually onto. In the case of GDS, topological transitivity, strong transitivity are preserved under semiconjugacy and very strong transitivity, topological mixing, locally eventually onto are preserved under strong semiconjugacy. The following is the schematic representation of implications among the notions of transitivity for non-autonomous discrete dynamical systems:\\

\hspace*{1.3cm}	LEO \hspace*{.5cm}$\implies$ \hspace*{.5cm}VST\hspace*{.5cm} $\implies$ 
\hspace*{.5cm}ST \hspace*{.5cm}$\implies$ \hspace*{.5cm}TT\\
\hspace*{8.1cm} $\big\Downarrow$ \hspace*{2.1cm} $\big\Downarrow$\\
\hspace*{5.9cm} Strong Extended  $\implies$  Extended\\ 
\hspace*{6.7cm}Transitivity \hspace*{1cm} Transitivity.\\
If each $f_n$ is surjective, we have\\

\hspace*{2cm}LEO \hspace*{3cm} $\implies$ \hspace*{3cm}TM \\
\hspace*{2.9cm}$\big\Downarrow$ \hspace*{7.7cm}$\big\Downarrow$ \\
\hspace*{.3cm}Strong Exact Transitivity \hspace*{0cm} $\implies$ \hspace*{0cm} Exact Transitivity \hspace*{0cm} $\implies$\hspace*{0cm} TT.

\section*{Acknowledgements}
The second author would like to thank UGC, INDIA for providing fellowship under NFOBC scheme [Ref. NO.:221610064995]. 

\end{document}